%% file: control_pure_neumann-FINAL.tex
\documentclass[10pt,final,reqno]{amsart}
\usepackage{amsfonts,latexsym,url,amsmath,amssymb,comment,enumitem}
\usepackage{upgreek}
\excludecomment{commentforus}
\usepackage{subfigure}
\usepackage{url,centernot}
\usepackage{graphicx,color}
\usepackage[color,notref,notcite]{showkeys}
\usepackage{bm}
\usepackage{stmaryrd}
\usepackage{bbm}
\usepackage{textcomp}
\usepackage{csquotes}

\DeclareMathAlphabet\mathbfcal{OMS}{cmsy}{b}{n}

\newcommand{\mathbi}[1]{{\boldsymbol #1}}

\def\R{\mathbb{R}}

\def\err{\mathsf{err}}

\def\P{\mathbfcal{P}}
\def\Proj{\P_{\!\mesh}}
\def\<{\langle}
\def\>{\rangle}

\def\div{{\rm div}}

\def\norm#1#2{\Vert#1\Vert_{#2}}

\newcounter{cst}

\def\bt{\begin{theorem}}
\def\et{\end{theorem}}
\def\bl{\begin{lemma}}
\def\el{\end{lemma}}
\def\bc{\begin{corollary}}
\def\ec{\end{corollary}}
\def\bd{\begin{definition}}
\def\ed{\end{definition}}
\def\br{\begin{remark}}
\def\er{\end{remark}}

\def\ad{{\rm ad}}

\def\Uad{\mathcal U_\ad}
\def\Uadh{\mathcal U_{\ad,h}}
\def\Uh{\mathcal U_{h}}

\def\bP{\overline{P}}

\def\cv{K}

\newcommand{\mesh}{{\mathcal T}}

\newcommand{\edge}{{\sigma}}

\newcommand{\edges}{{\mathcal E}}              
\newcommand{\edgescv}{{{\edges}_\cv}}

\newcommand{\x}{\mathbi{x}}

\newcommand{\bu}{\overline{u}}

\newcommand{\bvarphi}{\mathbi{\varphi}}

\newcommand{\be}{\begin{equation}}
\newcommand{\ee}{\end{equation}}
\newcommand{\disc}{{\mathcal D}}

\renewcommand{\O}{\Omega}

\def\dr{\partial}
\def\bfn{\mathbf{n}}

\renewcommand{\d}{{\:\rm d}}

\newcommand{\ba}{\begin{array}{llll}   }
\newcommand{\bac}{\begin{array}{c}}
\newcommand{\bari}{\begin{array}{r}}
\newcommand{\ea}{\end{array}}

\newcommand{\NORM}[1]{{\left\vert\kern-0.25ex\left\vert\kern-0.25ex\left\vert #1 
    \right\vert\kern-0.25ex\right\vert\kern-0.25ex\right\vert}}

\newcommand{\by}{\overline y}
\newcommand{\bp}{\overline  p}
\newcommand{\oc}{\overline c}
\newcommand{\tu}{\widetilde{u}}
\newcommand{\tp}{\widetilde{p}}

\newcommand{\yd}{\overline{y}_d}

\def\hat#1{\widehat{#1}}

\def\assum#1{{\rm\textbf{(A#1)}}}

\newtheorem{theorem}{Theorem}[section]
\newtheorem{remark}[theorem]{Remark}

\newtheorem{lemma}[theorem]{Lemma} 
\newtheorem{definition}[theorem]{Definition}
\newtheorem{proposition}[theorem]{Proposition}
\newtheorem{corollary}[theorem]{Corollary}

\numberwithin{equation}{section}
\def\WS{{\rm WS}}

\def\Xint#1{\mathchoice
{\XXint\displaystyle\textstyle{#1}}%
{\XXint\textstyle\scriptstyle{#1}}%
{\XXint\scriptstyle\scriptscriptstyle{#1}}%
{\XXint\scriptscriptstyle\scriptscriptstyle{#1}}%
\!\int}
\def\XXint#1#2#3{{\setbox0=\hbox{$#1{#2#3}{\int}$ }
\vcenter{\hbox{$#2#3$ }}\kern-.6\wd0}}

\def\dashint{\Xint-}

\usepackage[normalem]{ulem}
\newcounter{corr}
\definecolor{violet}{rgb}{0.580,0.,0.827}
\newcommand{\corr}[3]{\typeout{Warning : a correction remains in page
\thepage}
				\stepcounter{corr}        
				{\color{blue}\ifmmode\text{\,\sout{\ensuremath{#1}}\,}\else\sout{#1}\fi}
        {\color{red}#2}
        {\color{violet} \fbox{\thecorr}#3}}

\newcounter{cexp}
\catcode`\@=11
\def\terml#1{T_{\refstepcounter{cexp}\@bsphack
\protected@write\@auxout{}%
           {\string\newlabel{#1}{{\thecexp}{\thepage}}}\thecexp}}
\catcode`\@=12

\begin{document}
\title[Numerical analysis of pure Neumann control problem]{Numerical analysis for the pure Neumann control problem using the gradient discretisation method}
\author{J\'erome Droniou}
\address{School of Mathematical Sciences, Monash University, Clayton,
Victoria 3800, Australia.
\texttt{jerome.droniou@monash.edu}}
\author{Neela Nataraj}
\address{Department of Mathematics, Indian Institute of Technology Bombay, Powai, Mumbai 400076, India.
\texttt{neela@math.iitb.ac.in}}
\author{Devika Shylaja}
\address{IITB-Monash Research Academy, Indian Institute of Technology Bombay, Powai, Mumbai 400076, India.
\texttt{devikas@math.iitb.ac.in}}

\date{\today}

\maketitle

\begin{abstract}
The article discusses the gradient discretisation method (GDM) for distributed optimal control problems governed by diffusion equation with pure Neumann boundary condition. Using the GDM framework enables to develop an analysis that directly applies to a wide range of numerical schemes, from conforming and non-conforming
finite elements, to mixed finite elements, to finite volumes and mimetic finite differences methods.
 Optimal order error estimates for state, adjoint and control variables for low order schemes are derived under standard regularity assumptions. A {novel}{} projection relation between the optimal control and the adjoint variable allows the proof of a super-convergence result for post-processed control. Numerical experiments performed using a modified active set strategy algorithm for conforming, nonconforming and mimetic finite difference methods confirm the theoretical rates of convergence.
\end{abstract}

\section{Introduction}\label{sec:control_neumann}
Consider the following distributed optimal control problem governed by the diffusion equation with Neumann boundary condition:
\begin{subequations} \label{model.neumann}
	\begin{align}
	&  { \min_{u \in \Uad} J(y, u) \,\,\, \textrm{ subject to } }  \label{cost1}\\
	&  \quad{-\div(A\nabla y) =  u + f\,\, \mbox{ in } \Omega,  }\label{state1} \\
	&  \quad{A\nabla y\cdot\bfn_{\O}  =0\,\,\mbox{ on $\partial\O$},\qquad \dashint_{\O}y(\x)\d\x=0.} \label{state2}
	\end{align}
\end{subequations} 
Here,  $\Omega \subset {\mathbb R}^d \; (d \le 3)$ is a bounded domain with boundary $\partial \Omega$,
$\bfn_\O$ is the outer unit normal to $\O$ and $\dashint_{\O}y(\x)\d\x:=\frac{1}{|\O|}\int_\O y(\x)\d\x$ denotes the average value of the function $y$ over $\O$ (here and in the rest of the paper,
$|E|$ is the Lebesgue measure of a set $E\subset \R^d$). The cost functional, dependent on the
state variable $y$ and the control variable $u$, is given by
\begin{equation}\label{def:cost}
J(y, u) :=\frac{1}{2}\norm{y- \yd}{L^2(\O)}^2 + \frac{\alpha}{2} \norm{u}{L^2(\O)}^2
\end{equation}
with $\alpha>0$. The desired state variable $\yd\in L^2(\O)$ is chosen to satisfy $\dashint_\O \yd(\x)\d\x=0$. The source term $f \in L^2(\O)$ also satisfies
the zero average condition $\dashint_\O f(\x)\d\x=0$. The diffusion matrix $A:\O\rightarrow \mathcal M_d(\R)$
is a measurable, bounded and uniformly elliptic matrix-valued function such
that $A(\x)$ is symmetric for a.e. $\x\in \O$. Finally, the admissible set of controls $\Uad$ is the non-empty
convex set defined by
\be\label{comp.cond}
\Uad=\left\{u\in L^2(\O) \,:\,a\le u\le b \mbox{ and }\dashint_\O u(\x) \d\x=0\right\},
\ee
where $a$ and $b$ are constants in $[-\infty,+\infty]$ with $a<0<b$ (this condition is necessary to ensure that $\Uad$ is not empty or reduced to $\{0\}$).
\medskip

In this paper, we discuss the discretisation of control and state by the gradient discretisation method (GDM). The GDM is a generic framework for the convergence analysis of numerical methods (finite elements, mixed finite elements, finite volume, mimetic finite difference methods, etc.) for linear and non-linear elliptic and parabolic diffusion equations (including degenerate equations), the Navier-Stokes equations, variational inequalities, Darcy flows in fractured media, etc. See for example \cite{dro-12-gra,DEH15,DHM16,aln-14-vi,zamm2013}, and the monograph \cite{gdm} for a complete presentation of the GDM. The contributions of this article are summarised now. Here, we
\begin{itemize}
	\item establish basic error estimates that provide $\mathcal O(h)$ convergence rate for all the three variables (control, state and adjoint) for low order schemes under standard regularity assumptions for the pure Neumann problem, without reaction term. Given that the optimal control $\bu$ is approximated by piecewise constant functions, the convergence rates are optimal;
\item prove super-convergence result for post-processed optimal controls, state and adjoint variables; 
\item establish a projection relation (see Lemma \ref{u_proj_c}) between control and adjoint variables. This relation, which is non-standard since it has to account for the zero average constraints, is the key to prove the super-convergence result for all three variables;
\item design a  modified active set strategy algorithm for GDM that is adapted to this non-standard projection relation;
\item discuss some numerical results that confirm the theoretical rates of convergence for conforming, nonconforming and mimetic finite difference methods. 
\end{itemize} 
The literature contains several contributions to numerical analysis for second order distributed optimal control 
problems governed by diffusion equation with Dirichlet and Neumann boundary conditions (BC) (we refer to \cite{GDM_control, CMAR} and the references therein). For Dirichlet BC, the super-convergence of post-processed controls 
for conforming finite element (FE) methods has been investigated in \cite{CMAR}. Recently, this result was extended to 
the GDM in \cite{GDM_control}. Carrying out this analysis in the context of the GDM means that it readily applies to 
various schemes, including non-conforming $\mathbb{P}_1$ finite elements and hybrid mimetic mixed schemes (HMM), which 
contains mixed-hybrid mimetic finite differences; and for these schemes, the analysis of \cite{GDM_control} provides 
novel estimates. We refer to \cite{ECKK, ECMMJPR07, ECMMJPR, ECMMFT, ECJPR, ECFT} for an analysis of nonlinear 
elliptic optimal control problems. Several works cover optimal control for second order Neumann boundary value 
problems, albeit with an additional (linear or non-linear) reaction term which makes the state equation naturally well-
posed, without zero average constraint, see \cite{TAJPAR,TAJPAR15,ECMMFT,KP15,MMAR}. To the best of our knowledge, the 
numerical analysis of pure Neumann control problems, without reaction term and thus with the integral constraint, is 
open even for finite element methods. Our results therefore seem to be new and, being established in the GDM 
framework, cover a range of numerical methods, including conforming Galerkin methods, non-conforming finite elements, 
and mimetic finite differences. Although done on the simple problem \eqref{model.neumann},
the analysis uncovers some properties, such as the specific relation formula between the adjoint
and control variables and a modified active set algorithm used to compute the solution
of the numerical scheme. These have a wider application potential in optimisation problems involving an integral constraint. For example, in the model of \cite{cj86,pe66} describing the miscible displacement of one fluid by another in a porous medium, the pressure is subjected to an elliptic equation with homogeneous Neumann boundary
conditions. In this equation, the source terms model the injection and production wells, and are typically the only quantities that engineers can adjust (to some extent). Hence, considering these source terms as controls
of the pressure may lead to optimal control problems as in \eqref{model.neumann}, with the exact
same boundary conditions and integral constraints on the control terms.

\medskip

The paper is organised as follows. Section \ref{sec:GDM:ellptic} recalls the GDM for elliptic problems with Neumann BC and the properties needed to prove its convergence. Some classical examples of GDM are presented in Subsection \ref{subsec:eg}. Section \ref{sec:GDM:Control} deals with the GDM for the optimal control problem \eqref{model.neumann}. The basic error estimates and super-convergence results are presented in Subsections \ref{err_est} and \ref{supercv}. The first super-convergence result provides a nearly quadratic convergence rate for a post-proces\-sed control.
Under an $L^\infty$ stability assumption
of the GDM, which stems for most schemes from the quasi-uniformity of the mesh, the second super-convergence theorem establishes a full quadratic super-convergence rate. Discussions on post-proces\-sed controls and the projection relation between control and proper adjoint are presented in Subsection \ref{dis:ppu}. The active set strategy is an algorithm to solve the non-linear Karush-Kuhn-Tucker (KKT) formulation of the optimal control problem \cite{tf}. Subsection \ref{activeset} presents a modification of this algorithm that accounts for the zero average constraint on the control. This modified active set algorithm also automatically selects the proper discrete adjoint whose projection provides the discrete control variable. In Subsection \ref{examples}, we present the results of some numerical experiments. The paper ends with an appendix, Section \ref{sec:appendix}, where we prove the results stated in Subsections \ref{err_est} and \ref{supercv}.

\medskip
 
Before concluding this introduction, we discuss the optimality conditions for \eqref{model.neumann}. For a given $u \in \Uad$, there exists a unique weak solution $y(u) \in H^1_\star(\Omega):=\{w\in H^1(\O)\,:\dashint_\O w(\x) \d\x=0\}$ of \eqref{state1}--\eqref{state2}. That is, for $u \in \Uad$, there exists a unique  $y(u) \in H^1_\star(\Omega)$ such that for all ${w}\in H^1_\star(\O)$,
\begin{equation}\label{weak_state}
a(y,w)=\int_\O u {w} \d\x ,
\end{equation}
where $a(\phi,\psi)=
\int_\O A\nabla \phi\cdot\nabla \psi \d\x$ for all $\phi, \psi \in H^1(\O)$. The term $y(u)$ is the {state} {associated} with the control $u$.

In the following, the norm and scalar product in $L^2(\O)$ (or $L^2(\O)^d$ for vector-valued functions) are denoted by $\|\cdot\|$ and $(\cdot,\cdot)$. The convex control problem \eqref{model.neumann} has a unique solution $(\by,\bu) \in H^1_\star(\O) \times \Uad $ and there exists a co-state $\bp \in H^1(\O)$ such that the triplet $(\by, \bp, \bu) \in H^1_\star(\O) \times H^1(\O) \times \Uad$ satisfies the KKT optimality conditions \cite[Chapter 2]{jl}:
\begin{subequations} \label{continuous_kkt}
	\begin{align}
	& a(\by,w) = (\bu +f, w) \,  &\   \forall \: w \in H^1_\star(\O), \label{state_cont} \\
	& a(z,\bp) = (\by-\yd, z)\,  &\   \forall \: z \in H^1(\O), \label{adj_cont}\\
	& (\bp+\alpha\bu,v-\bu)\geq 0\,  &\  \forall \: v\in  \Uad. \label{opt_cont}
	\end{align}	
\end{subequations}

\begin{remark}[Choice of the adjoint state]
{Several co-states satisfy the optimality conditions \eqref{continuous_kkt}, as $\bp$ is only determined up to an additive constant by \eqref{continuous_kkt}. The same will be true for the discrete co-state,
solution to a discrete version of these KKT equations. Establishing error estimates require the continuous
and discrete co-states to have the same average. The usual choice is to fix this average as zero.
However, for the control problem with pure Neumann conditions, this is not the best choice. Indeed, as seen in Lemma \ref{u_proj_c}, 
establishing a proper relation between the control and co-state requires a certain zero average of a
\emph{non-linear} function of this co-state. A more efficient
approach, that we will adopt, to fix the proper co-states is thus the following:
\begin{enumerate}
\item Design an algorithm (the modified active set algorithm of Subsection \ref{activeset}) that computes
a discrete co-state with the proper condition, so that the discrete control can be easily
obtained in terms of this discrete co-state,
\item Fix the average of the continuous co-state $\bp$ to be the same as the average of the
discrete co-state obtained above.
\end{enumerate}
As we will see,  an algebraic relation between this $\bp$ and the continuous control $\bu$ can still be written,
upon selecting a proper (but non-explicit) translation of $\bp$.}
\end{remark}

\begin{remark}[Zero average constraint on the source term and desired state]\label{avg_f}
 \noindent
\begin{itemize}
	\item[(i)] 
If we consider \eqref{model.neumann} without the constraint $\dashint_\O f(\x)\d\x = 0$ on the source term, the set of admissible controls needs to be modified into
\[
\Uad=\left\{u\in L^2(\O) \,:\,a\le u\le b \mbox{ and }\dashint_\O (u+f) \d\x=0\right\}.
\]
In this case, a simple transformation can bring us back to the case of a source term
with zero average. Rewrite the state equation \eqref{state1} as $-\div(A\nabla y)=u^\star + f^\star$ with $u^\star=u+\dashint_\O f \d\x$ and $f^\star=f-\dashint_\O f\d\x$. Then, $\dashint_\O f^\star\d\x=0$ and $u^\star\in \Uad^\star$ where
\[
\Uad^\star=\left\{u^\star\in L^2(\O) \,:\,a^\star\le u^\star\le b^\star \mbox{ and }\dashint_\O u^\star \d\x=0\right\}
\]
with $a^\star=a+\dashint_\O f\d\x$ and $b^\star=b+\dashint_\O f\d\x$.
\item[(ii)] If the desired state $\yd\in L^2(\O)$ is such that $\dashint_\O \yd(\x)\d\x =:m\neq 0$, 
then it is natural to select states $y$ in \eqref{model.neumann} with the same average $m$ (since
the average of these states can be freely fixed, and the choice made in \eqref{state2} is arbitrary).
This ensures the best possible approximation of the desired state $\yd$. In that case, working with
$y-m$ and $\yd-m$ instead of $y$ and $\yd$ brings back to the original formulation \eqref{model.neumann}
with a desired state $\yd-m$ having a zero average.
\end{itemize}
\end{remark}

\begin{remark}[Non-homogeneous BCs]\label{non-homBC}
The study of second order distributed control problem \eqref{model.neumann} with non-homogeneous boundary conditions 
$A\nabla y\cdot\bfn_{\O}  =g$ on $\partial\O$ (with $g \in  L^2(\partial\O)$)
follows in a similar way. In this case, the source terms and boundary condition are supposed to satisfy
the compatibility condition
\[
\int_\O f\d\x+\int_{\partial\O}g\d s(\x)=0.
\]
The controls are still taken in $\Uad$ defined by \eqref{comp.cond} and the KKT optimality condition is \cite{jl}: Seek $(\by, \bp, \bu) \in H^1_\star(\O) \times H^1(\O) \times \Uad$ such that
	 	\begin{align*}
	 	& a(\by,w) = (\bu + f, w) + (g,\gamma(w))_\dr\,  &\   \forall \: w \in H^1_\star(\O), \\
	 	& a(z,\bp) = (\by-\yd, z)\,  &\   \forall \: z \in H^1(\O), \\
	 	& (\bp+\alpha\bu,v-\bu)\geq 0\,  &\  \forall \: v\in  \Uad,
	 	\end{align*}	
	 where $\gamma:H^1(\O) \rightarrow L^2(\dr\O)$ is the trace operator and $(\cdot,\cdot)_\partial$ is the inner product
in $L^2(\partial\O)$.
\end{remark}


\section{GDM for elliptic PDE with Neumann BC}\label{sec:GDM:ellptic}
The gradient discretisation method (GDM) consists in writing numerical schemes,
called gradient schemes (GS), by replacing the continuous space and operators by discrete ones in the weak formulation of the problem \cite{gdm,dro-12-gra,eym-12-sma}. These discrete space and operators are given by a gradient discretisation (GD).

\subsection{Gradient discretisation and gradient scheme} \label{subsec:GD.GS}

A notion of gradient discretisation for Neumann BC is given in \cite[Definition 3.1]{gdm}. The following extends this definition by demanding the existence of the element $1_\disc$ and is always satisfied in practical applications. This existence ensures that the zero average condition can be put in the discretisation space or in the bilinear form as for the continuous formulation, see Remark \ref{discrete_space}.

\begin{definition}[Gradient discretisation for Neumann boundary  conditions]\label{def:GD}
A gradient discretisation (GD) for homogeneous Neumann boundary conditions is given by $\disc=(X_{\disc},\Pi_\disc,\nabla_\disc)$ such that
	\begin{itemize}
		\item $X_{\disc}$ is a finite dimensional vector space on $\R$.
		\item $\Pi_\disc:X_{\disc}  \rightarrow L^2(\O)$ and $\nabla_\disc:X_{\disc}  \rightarrow L^2(\O)^d$ are linear mappings.
		\item The quantity
		\be\label{def:norm}
		\norm{w}{\disc}^2:= \norm{\nabla_\disc w}{}^2+\left|\dashint_\O\Pi_\disc w(\x)\d\x\right|^2
		\ee
		is a norm on $X_{\disc}$.
		\item There exists $1_\disc\in X_\disc$ such that $\Pi_\disc 1_\disc = 1$ on $\O$ and $\nabla_\disc 1_\disc=0$ on $\O$.
	\end{itemize}
\end{definition}

The flexibility of the GDM analysis framework comes from the wide possible range of choices for $(X_\disc,\Pi_\disc,\nabla_\disc)$. Each of these choices correspond to a particular numerical scheme (see Subsection \ref{subsec:eg}
for a few examples). The space $X_\disc$ represents the degrees of freedom (unknowns) of the method; a vector in $X_\disc$ gathers values for such unknowns. The operators $\Pi_\disc$ and $\nabla_\disc$ reconstruct, from a set of values of these unknowns, a scalar (resp. vector) function on the entire set $\O$. The scalar function is suppose to play the role of the 
solution/test functions itself in the weak formulation of the PDE; the vector-function, reconstructed ``gradient'', is used in lieu of the gradients of these solution/test functions. Performing these substitutions in the weak formulation leads to a finite-dimensional system of equations (on the unknowns), which is dubbed the gradient scheme corresponding to
the gradient discretisation $\disc$.

If $F\in L^2(\O)$ is such that $\dashint_\O F(\x)\d\x=0$, the weak formulation of the Neumann boundary value problem
\be\label{base}
\left\{
\ba
-\div(A\nabla \psi)=F\mbox{ in $\O$},\\
A\nabla \psi\cdot\bfn_\O=0\mbox{ on $\dr\O$}
\ea
\right.
\ee
is given by
\be\label{base_weak}
\mbox{Find $\psi \in H^1_\star(\O)$ such that, for all $w \in H^1_\star(\O)$, }
a(\psi,w) =	\int_\O F w \d\x.
\ee
As explained above, a gradient scheme for \eqref{base} is then obtained from a GD $\disc$ by writing the
weak formulation \eqref{base_weak} with the continuous spaces, functions and gradients
replaced with their discrete counterparts:
\be\label{base.GS}
\begin{aligned}
	&\mbox{Find $\psi_\disc\in X_{\disc,\star}$ such that, for all $w_\disc\in X_{\disc,\star}$,}\\
	&a_{\disc}(\psi_{\disc},w_{\disc})=
	\int_\O F\Pi_\disc w_\disc\d\x,
\end{aligned}
\ee
where $\displaystyle a_{\disc}(\phi_{\disc},z_{\disc})=
\int_\O A\nabla_{\disc} \phi_{\disc}\cdot\nabla_{\disc} {z}_{\disc} \d\x$, for all 
$\phi_{\disc}, z_{\disc} \in X_{\disc}$, and $X_{\disc,\star}=\{w_\disc\in X_\disc\,:\,\dashint_\O \Pi_\disc w_\disc\d\x=0\}$.

\begin{remark} \label{discrete_space}
As for the continuous formulation \eqref{base_weak}, using the element $1_\disc\in X_\disc$ actually enables us to consider in \eqref{base.GS} test functions $w_\disc$ in $X_{\disc}$, rather than just $X_{\disc,\star}$. The simplest technique to achieve this is to use a quadratic penalty method \cite[Chapter 11]{MSG_FEM}. For any $\rho>0$, \eqref{base.GS} can be shown equivalent to
	\be\label{base.GS.remark}
	\begin{aligned}
		&\mbox{Find $\psi_\disc\in X_{\disc}$ such that, for all $w_\disc\in X_{\disc}$,}\\
		&a_{\disc}(\psi_{\disc},w_{\disc}) + \rho \left(\dashint_\O\Pi_\disc \psi_\disc
		\d\x\right)\left(\dashint_\O \Pi_\disc w_\disc\d\x\right)
		=\int_\O F\Pi_\disc w_\disc\d\x.
	\end{aligned}
	\ee
{Indeed, considering $w_\disc=1_\disc$ in \eqref{base.GS.remark} shows that the solution
to this problem belongs to $X_{\disc,\star}$ and is therefore a solution to \eqref{base.GS}. The
converse is straightforward.}
\end{remark}

\subsection{Examples of gradient discretisations} \label{subsec:eg}

We briefly present here a few examples based on known numerical methods. We refer to \cite{gdm}
for a detailed analysis of these methods, and more examples of gradient discretisations. As demonstrated by these examples,
the GDM cover a wide range of different numerical methods. This means that the analysis carried out
in the GDM framework for the control problem in \eqref{model.neumann} readily applies to all
these methods. In particular, this makes the control problem accessible to numerical schemes
not usually considered but relevant to diffusion models, such as schemes applicables on generic meshes
(not just triangular/quadrangular meshes) as encountered for example in reservoir engineering applications.

\medskip

Let us consider a mesh $\mesh$ of $\O$. A precise definition can be found for example in \cite[Definition 7.2]{gdm}
but, for our purpose, the intuitive understanding of $\mesh$ as a partition of $\O$ in polygonal/polyhedral sets
is sufficient.

\medskip

\textbf{Conforming $\mathbb{P}_1$ finite elements}. The simplest gradient discretisation is perhaps
obtained by considering conforming $\mathbb{P}_1$ finite elements. The mesh is made of triangles
(in 2D) or tetrahedra (in 3D), with no hanging nodes.
Each $v_\disc\in X_{\disc}$ is a vector of values at the vertices of the mesh (the standard unknowns
of conforming $\mathbb{P}_1$ finite elements). $\Pi_\disc v_\disc$ is the continuous piecewise linear function on the mesh which takes these values at the vertices, and $\nabla_\disc v_\disc= \nabla(\Pi_\disc v_\disc)$.
Then \eqref{base.GS} is the standard $\mathbb{P}_1$ finite element scheme for \eqref{base}.

\medskip

\textbf{Non-conforming $\mathbb{P}_1$ finite elements}. As above, the mesh is made of conforming
triangles or tetrahedra. Each $v_\disc\in X_{\disc}$  is a vector of values at the centers of mass of the edges/faces,
$\Pi_\disc v_\disc$ is the piecewise linear function on the mesh which takes these values at these centers of mass,
and $\nabla_\disc v_\disc= \nabla_\mesh(\Pi_\disc v_\disc)$ is the broken gradient of $\Pi_\disc v_\disc$.
In that case, \eqref{base.GS} gives the non-conforming $\mathbb{P}_1$ finite element approximation of \eqref{base}.

\medskip

\textbf{Mass-lumped non-conforming $\mathbb{P}_1$ finite elements}. Still considering a conforming
triangular/tetrahedral mesh, the space $X_\disc$ and gradient reconstruction
$\nabla_\disc$ are identical to those of the non-conforming $\mathbb{P}_1$ finite elements described above, but
the function reconstruction is modified to be piecewise constant. For each edge/face $\edge$ of the mesh, we consider
the diamond $D_\edge$ around $\edge$ constructed from the edge/face and the one or two cell centers on each side
(see Figure \ref{fig:diamond}). Then, for $v=(v_\edge)_{\edge}\in X_\disc$, the reconstructed
function $\Pi_\disc v$ is the piecewise constant function on the diamonds, equal to $v_\edge$ on $D_\edge$ for all edge/face
$\edge$.

\begin{figure}[htb]
\begin{center}
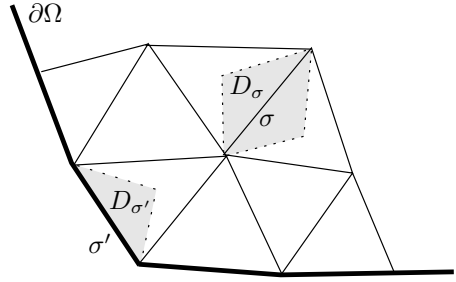
\caption{Diamonds for the definition of the mass-lumped non-conforming $\mathbb{P}_1$ finite elements}
\label{fig:diamond}
\end{center}
\end{figure}

\medskip

\textbf{Hybrid mixed mimetic method (HMM)}. We now consider a generic mesh $\mesh$ (not necessarily
triangular/tetrahedral) with one
point $\x_K$ chosen in each cell $K\in\mesh$ such that $K$ is strictly star-shaped with respect to $\x_K$; see Figure
\ref{fig:mesh} for some notations. A vector $v\in X_{\disc}$ is made of cell $(v_K)_K$ and face $(v_\edge)_\edge$ values,
and the operator $\Pi_\disc$ reconstructs a piecewise constant function from the cell values:
for any cell $K$, $\Pi_\disc v = v_K$ on $K$. The gradient reconstruction $\nabla_\disc$
is built in two pieces: a consistent gradient $\overline{\nabla}_K$ constant over the cell
and stabilisation terms constant over the half-diamonds $D_{K,\edge}$ (and akin to the
remainders of first-order Taylor expansions between the cell and face values). For any cell $K$
and any face $\edge$ of $K$, we set
\[
\nabla_\disc v = \overline{\nabla}_K v + \frac{\sqrt{d}}{d_{K,\edge}}\left(v_\edge-v_K-\overline{\nabla}_K v\cdot(
\overline{\x}_\edge-\x_K)\right)\bfn_{K,\edge}\quad\mbox{ on $K$},
\]
where $d_{K,\edge}$ is the orthogonal distance between $\x_K$ and $\edge$, $\overline{\x}_\edge$ is the center of mass of $\edge$, $\bfn_{K,\edge}$ is the outer normal to $K$ on $\edge$ and, denoting by $\edgescv$ the set of faces of $K$,
\[
\overline{\nabla}_K v = \frac{1}{|K|}\sum_{\edge\in\edgescv}|\edge|v_\edge\bfn_{K,\edge}.
\]
Once used in the gradient scheme \eqref{base.GS}, this HMM gradient discretisation gives rise to
a numerical method that can be applied on any mesh (including with hanging nodes, non-convex cells, etc.).
This scheme can also be re-interpreted as a finite volume method \cite[Section 13.3]{gdm}.
\begin{figure}[htb]
\begin{center}
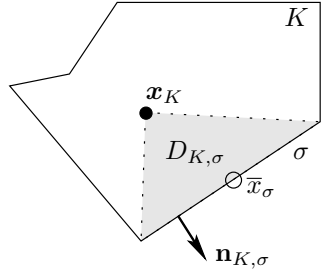
\caption{Notations for the construction of the HMM gradient discretisation}
\label{fig:mesh}
\end{center}
\end{figure}

\subsection{Error estimates for the GDM for the Neumann problem}\label{subsec:errorpde}

The accuracy of a gradient scheme \eqref{base.GS} is measured by three quantities. The first one is a discrete Poincar\'e constant $C_\disc$, which ensures the \emph{coercivity} of the method.
\be\label{def.CD}
C_\disc := \max_{w\in X_{\disc}\setminus \{0\}} \frac{\norm{\Pi_\disc w}{}}{\norm{w}{\disc}}.
\ee
The second quantity is the interpolation error $S_\disc$, which
measures what is called, in the GDM framework, the \emph{GD-consistency} of $\disc$.
\be\label{def.SD}
\forall \varphi\in H^1(\O)\,,\;
S_\disc(\varphi)=\min_{w\in X_{\disc}}\left(\norm{\Pi_\disc w-\varphi}{}
	+\norm{\nabla_\disc w-\nabla\varphi}{}\right).
\ee

Finally, we measure the \emph{limit-conformity} of a GD by defining
\be\label{def.WD}
\begin{aligned}
	&\forall \bvarphi\in H^{\div}_0(\O)\,,\\
	&W_\disc(\bvarphi)=\max_{w\in X_{\disc}\setminus \{0\}} \frac{1}{\norm{w}{\disc}}
		\left|\int_\O (\Pi_\disc w \div(\bvarphi)	+\nabla_\disc w\cdot\bvarphi)\d\x \right|,
\end{aligned}
\ee
where $H^{\div}_0(\O)=\{\bvarphi\in L^2(\O)^d\,:\,\div(\bvarphi)\in L^2(\O)\,,\;
\gamma_{\bfn}(\bvarphi)=0\}$ with $\gamma_{\bfn}$ being the normal trace of $\bvarphi$ on $\dr\O$.

Using these quantities, an error estimate can be established for GS. We refer to  \cite{gdm} for a proof of the following theorem. Here and in the rest of the paper,
\begin{equation}\label{notation:lesssim}
\begin{aligned}
&X\lesssim Y\mbox{ means that }X\le CY\mbox{ for some $C$ depending}\\
&\mbox{only on $\O$, $A$ and an upper bound of $C_\disc$}.
\end{aligned}
\end{equation}

\begin{theorem}[Error estimate for the GDM] \label{th:error.est.PDE} Let $\disc$ be a GD in the
	sense of Definition \ref{def:GD}, let $\psi$ be the solution
	to \eqref{base_weak}, and let $\psi_\disc$  be the solution to \eqref{base.GS}. Then
	\be\label{est.PiD.nablaD}
	\norm{\Pi_\disc \psi_\disc-\psi}{}+\norm{\nabla_\disc \psi_\disc-\nabla\psi}{}
	\lesssim \WS_\disc(\psi),
	\ee
	where
	\be \label{def.ws}
	\WS_\disc(\psi)= W_\disc(A\nabla \psi)+S_\disc(\psi).
	\ee
\end{theorem}

\begin{remark}[Rates of convergence] \label{rate.ellptic}
	For all classical low order methods based on meshes (such as $\mathbb{P}_1$ conforming
	and non-conforming finite element methods, finite volume methods, etc.), if $A$ is Lipschitz continuous and $\psi\in H^2(\O)$ then $\mathcal O(h)$ estimates can be obtained for $W_\disc(A\nabla \psi)$ and $S_\disc(\psi)$ \cite{gdm}.
	Theorem \ref{th:error.est.PDE} then gives a linear rate of convergence
	for these methods.
\end{remark}

\begin{remark} \label{avg_err_estimate}
	Note that Theorem \ref{th:error.est.PDE} also holds if we replace the zero average condition on $\psi$ and $\Pi_\disc \psi_\disc$ with $\dashint_\O \Pi_\disc \psi_\disc \d\x=\dashint_\O \psi \d\x.$ In this case, the estimate \eqref{est.PiD.nablaD} can be obtained by considering the translation of $\psi_\disc$ and $\psi$. Set $\widetilde{\psi}_\disc=\psi_\disc-c1_\disc $ and $\widetilde{\psi}=\psi-c1, $  where $c=\dashint_\O \Pi_\disc \psi_\disc \d\x=\dashint_\O \psi \d\x$ and $1$ is the constant function. Using Definition \ref{def:GD}, we find $\Pi_\disc \widetilde{\psi}_\disc=\Pi_\disc {\psi}_\disc-c$, $\nabla_\disc \widetilde{\psi}_\disc=\nabla_\disc {\psi}_\disc$ and $\nabla \widetilde{\psi}=\nabla {\psi}$. This gives $\dashint_\O \Pi_\disc \widetilde{\psi}_\disc \d\x=\dashint_\O \widetilde{\psi} \d\x=0.$ Applying Theorem \ref{th:error.est.PDE},
$$	\norm{\Pi_\disc \widetilde{\psi}_\disc-\widetilde{\psi}}{}+\norm{\nabla_\disc \widetilde{\psi}_\disc-\nabla\widetilde{\psi}}{}
	\lesssim \WS_\disc(\widetilde{\psi})$$
which implies
$$	\norm{\Pi_\disc {\psi}_\disc-{\psi}}{}+\norm{\nabla_\disc {\psi}_\disc-\nabla{\psi}}{}
	\lesssim \WS_\disc(\widetilde{\psi})=\WS_\disc({\psi}).$$
	
\end{remark}

The following stability result, useful to our analysis, is straightforward.

\begin{proposition}[Stability of the GDM]\label{prop.stab}
	Let $\underline{a}$ be a coercivity constant of $A$.
	If $\psi_\disc$ is the solution to the gradient scheme \eqref{base.GS}, then
	\be\label{stability}
	\norm{\nabla_\disc \psi_\disc}{}
	\le \frac{C_\disc}{\underline{a}}\norm{F}{}
\quad\mbox{	and  }\quad
	\norm{\Pi_\disc \psi_\disc}{} \le 
	\frac{C_\disc^2}{\underline{a}}\norm{F}{}.
	\ee
\end{proposition}

\begin{proof}
	Take $w_\disc=\psi_\disc$ in \eqref{base.GS} and use the
	definition of $C_\disc$ to write
	\begin{align*}
	\underline{a}
	\norm{\nabla_\disc \psi_\disc}{}^2
	&\le \norm{F}{}\norm{\Pi_\disc \psi_\disc}{}
	\le C_\disc \norm{F}{}\norm{\psi_\disc}{\disc}.
	\end{align*}
	Since $\dashint_\O\Pi_\disc \psi_\disc\d\x=0$, recalling the Definition \eqref{def:norm}
	of $\norm{\cdot}{\disc}$ shows that $\norm{\psi_\disc}{\disc}=\norm{\nabla_\disc\psi_\disc}{}$
	and the proof of first estimate is complete. The second estimate follows from the
	definition of $C_\disc$. \end{proof}

\section{ GDM  for the control problem and main results}\label{sec:GDM:Control}

This section starts with a description of GDM for the optimal control problem and is followed by the basic error estimates and super-convergence results in Subsections \ref{err_est} and \ref{supercv}. The super-convergence results  for a post-processed control are presented.  
A discussion on post-processed controls and the projection relation between control and proper adjoint are presented in Subsection \ref{dis:ppu}. 

\subsection{GDM for the optimal control problem}\label{GDM:Control.def}

Let $\disc$ be a GD as in Definition \ref{def:GD}. The space $\Uh$ is defined as the space of piecewise constant functions on a mesh $\mesh$ of $\O$. The space $\Uadh=\Uad\cap\Uh$ is a finite dimensional subset of $\Uad$. A gradient scheme for \eqref{continuous_kkt} consists in seeking $(\by_{\disc}, \bp_{\disc}, \bu_{h})\in X_{\disc,\star} \times X_{\disc} \times \Uadh$, such that
\begin{subequations} \label{discrete_kkt}
	\begin{align}
	& a_{\disc}(\by_{\disc},w_{\disc}) = ( \bu_{h} + f, \Pi_\disc w_{\disc})
	&\forall \: w_{\disc} \in  X_{\disc,\star},  \label{discrete_state} \\
	& a_{\disc}(z_{\disc},\bp_{\disc}) = (\Pi_\disc \by_{\disc}-\yd, \Pi_\disc z_{\disc})&  \forall \: z_{\disc} \in  X_{\disc}, \label{discrete_adjoint} \\
	&(\Pi_\disc \bp_{\disc} +\alpha \bu_h,v_h-\bu_h)\geq  0
	&\forall \: v_h \in  \Uadh. \label{opt_discrete}
	\end{align}	
\end{subequations} 
As in the continuous KKT conditions \eqref{continuous_kkt}, these equations do not define
$\bp_\disc$ uniquely. One possible constraint that fixes a unique $\bp_\disc$ is described in Lemma
\ref{u_proj_c}. This particular choice ensures a simple projection relation between $\bp_{\disc}$ and $\bu_h$.

 Two projection operators play a major role throughout the paper: the orthogonal projection on piecewise constant functions on $\mesh$, namely $\Proj:L^1(\O)\to \Uh$ and the cut-off function $P_{[a,b]}:\R\to [a,b]$. They are defined as
 \be \label{def_projmesh}
 \forall v\in L^1(\O)\,,\:\forall K\in\mesh\,,\quad
 (\Proj v)_{|K} :=\dashint_K v(\x)\d\x\,,
 \ee
 \be \label{def_proj[a,b]}
 \forall s\in\R\,,\quad P_{[a,b]}(s) := \min(b, \max(a, s)).
 \ee
 
\subsection{Basic error estimate for the GDM for the control problem}\label{err_est}

To state the error estimates, we define the projection error $E_h$ by
\be \label{triple.norm}
\forall \: W\in L^2(\O)\,,\;E_h(W)=\norm{W-\Proj W}{}
\ee

{The proofs of the following basic error estimates are provided in Section \ref{sec:appendix}.
They follow by adapting the corresponding proofs in \cite{GDM_control} to account for the pure Neumann
boundary conditions and integral constraints.}

\begin{theorem}[Control estimate] \label{theorem.control}
	Let $\disc$ be a GD, $(\by,\bp,\bu)$ be a solution to \eqref{continuous_kkt}
	and $(\by_\disc,\bp_\disc,\bu_h)$ be a solution to \eqref{discrete_kkt} such that $\dashint_{\O}^{}\Pi_\disc \bp_\disc \d\x =\dashint_{\O}^{} \bp\d\x$. Then, recalling \eqref{notation:lesssim}, \eqref{def.ws} and \eqref{triple.norm}, there exists a constant $C$ depends only on $\alpha$ such that 
	\be\label{est.basic.u}
	\begin{aligned}
		\norm{\bu-\bu_h}{}\lesssim{}& C\left(E_h(\bp)+E_h(\bu)+\WS_\disc(\bp)+\WS_\disc(\by)\right).
	\end{aligned}
	\ee
\end{theorem}

\begin{proposition}[State and adjoint error estimates]\label{prop.state.adj}
	Let $\disc$ be a GD, 
	$(\by, \bp, \bu)$ be a solution to \eqref{continuous_kkt} and
	$(\by_\disc, \bp_\disc, \bu_h)$ be a solution to \eqref{discrete_kkt}. Assume that $\dashint_{\O}^{}\Pi_\disc \bp_\disc \d\x =\dashint_{\O}^{} \bp\d\x$. Then the following error estimates hold:
	\begin{align}
	\label{est.basic.y}
	\norm{\Pi_\disc \by_\disc-\by}{}+\norm{\nabla_\disc \by_\disc-\nabla\by}{}   \lesssim{}& \norm{\bu-\bu_h}{}+ \WS_\disc(\by),\\
	\label{est.basic.p}
	\norm{\Pi_\disc \bp_\disc-\bp}{}+\norm{\nabla_\disc \bp_\disc-\nabla\bp}{} \lesssim{}& \norm{\bu-\bu_h}{} + \WS_\disc(\by) + \WS_\disc(\bp).
	\end{align}		
\end{proposition}
\begin{remark}[Rates of convergence for the control problem]\label{rates.control}
As in Remark \ref{rate.ellptic}, if $A$ is Lipschitz continuous and $(\by, \bp, \bu)\in H^2(\O)^2 \times H^1(\O)$ then \eqref{est.basic.u}, \eqref{est.basic.y} and \eqref{est.basic.p} give linear rates of convergence for all classical first-order methods.
\end{remark}
\subsection{Super-convergence for post-processed controls} \label{supercv}

In this subsection, we define the post-processed continuous and discrete controls (see \eqref{Projection_ppu}) and state the super-convergence results. The proofs are presented in Section \ref{sec:appendix}.

We make here the following assumptions.
\begin{itemize}
	\item[\assum{1}][\emph{Interpolation operator}]
	For each $w\in H^2(\O)$, there exists $w_\mesh\in L^2(\O)$ such that:\\
\noindent	i) If $w \in H^2(\O)$ solves $-\div(A\nabla w)=g\in H^1(\O)$, and $w_\disc$ is the solution to the corresponding GS
	with $\dashint_\O\Pi_\disc w_\disc\d\x=\dashint_\O w\d\x$, then
	\be\label{state:scv}
	\norm{\Pi_\disc w_\disc-w_\mesh}{}  \lesssim  h^2\norm{g}{H^1(\O)}.
	\ee
\noindent	ii) For any $w\in H^2(\O)$, it holds
	\be\label{prop.M.1}
	\begin{aligned}
		\forall v_\disc \in X_{\disc}\,,\;
		\big|(w-w_\mesh,\Pi_\disc v_\disc)\big|
		\lesssim   h^2\norm{w}{H^2(\O)}\norm{\Pi_\disc v_\disc}{}\,,
	\end{aligned}
	\ee
	\be\label{prop.M.2}
	\norm{\Proj(w-w_\mesh)}{}\lesssim   h^2\norm{w}{H^2(\O)}.
	\ee
	\item[\assum{2}] The estimate $ \norm{\Pi_\disc v_\disc-\Proj(\Pi_\disc v_\disc)}{}  \lesssim  h\norm{\nabla_\disc v_\disc}{}$ holds for any $v_\disc\in X_{\disc}$.
	\item[\assum{3}] [\emph{Discrete Sobolev imbedding}] For all $v_\disc \in X_{\disc}$, it holds
	\[
	\norm{\Pi_\disc v_\disc}{L^{2^*}(\O)} \lesssim   \norm{ v_\disc}{\disc},
	\]
	where $2^*$ is a Sobolev exponent of 2, that is, $2^*\in [2,\infty)$ if $d=2$, and
	$2^*=6$ if $d=3$.
\end{itemize}

Let
\[
	\mesh_2=\{K\in\mesh\,:\,\mbox{$\bu=a$ a.e.\ on $K$, or $\bu=b$ a.e. on $K$,
		or $a<\bu<b$ a.e.\ on $K$}\} ,
\]
and $\mesh_1=\mesh\setminus\mesh_2$. That is, $\mesh_1$ is the set of cells where $\bu$ crosses at least
one constraint $a$ or $b$.
For $i=1,2$, we let $\O_{i,\mesh}={\rm int}(\cup_{K\in\mesh_i}\overline{K})$.
The space $W^{1,\infty}(\mesh_1)$ is the usual broken Sobolev space,
endowed with its broken norm. Our last assumption is:
\begin{itemize}
	\item[\assum{4}] $|\O_{1,\mesh}|\lesssim h$ and $\bu_{|\O_{1,\mesh}}\in W^{1,\infty}(\mesh_1)$.
\end{itemize}
\medskip

Possible choices of the mapping $w\mapsto w_\mesh$, depending on the considered gradient discretisation (that is, the considered numerical method), is discussed in Remark \ref{remark:supercv} below. Note that the assumptions \assum{1}--\assum{4} are similar to that in \cite{GDM_control} with $X_{\disc,0}$ substituted by $X_{\disc}$, and an additional average condition in $\assum{1}$. We also refer to \cite{GDM_control} for a detailed discussion on \assum{1}--\assum{4}.

\medskip

Assuming $\bp \in H^2(\O)$ (see Theorem \ref{thm.super-convergence}) and letting $\bp_\mesh$ be defined as in \assum{1}, the post-processed continuous and discrete controls are given by
\begin{equation} 
\begin{aligned}
\tu(\x)={}&P_{[a,b]}\left(-\frac{1}{\alpha}\bp_\mesh(\x) \right)\quad\mbox{ and }\quad
\tu_{h}(\x)={}P_{[a,b]}\left(-\frac{1}{\alpha} \Pi_\disc \bp_\disc(\x)\right).
\end{aligned}
\label{Projection_ppu}
\end{equation}
For a detailed discussion on the post-processed controls, we refer the reader to Subsection \ref{dis:ppu}.

\begin{remark}[Choice of $w_\mesh$ in \assum{1}] \label{remark:supercv}
When  $\Pi_\disc$ is a piecewise linear reconstruction, the super-convergence result \eqref{state:scv} usually holds with $w_\mesh=w$. This is for example well-known for conforming and non-conforming $\mathbb{P}_1$ FE method. When $\Pi_\disc v_\disc$ is piecewise constant on $\mesh$ for all $v_\disc\in X_{\disc}$, the super-convergence \eqref{state:scv} requires to project the exact solution on piecewise constant functions on the mesh. This is usually done by setting $w_\mesh(\x)=\dashint_K w(\x)\d\x$ for all $\x\in K$ and all $K\in\mesh$ (or, equivalently at order $\mathcal O(h^2)$, $w_\mesh(\x)=w(\overline{\x}_K)$
with $\overline{\x}_K$ the center of mass of $K$). With this choice, the super-convergence result is known, e.g., for mixed/hybrid and nodal mimetic finite difference schemes (see \cite{mimeticfdm,jd_nn}).
As a consequence,
		\begin{itemize}
			\item For FE methods, $\tu=P_{[a,b]}(-\alpha^{-1}\bp)$.
			\item For mimetic finite difference methods, $\tu_{|K}=P_{[a,b]}(-\alpha^{-1}\bp(\overline{\x}_K))$
for all $K\in\mesh$.
		\end{itemize}	
	\end{remark}

For $K\in\mesh$ and $\overline{\x}_K$ the center of mass of $K$, let $\rho_K=\max\{r>0\,:\,B(\overline{\x}_K,r)\subset K\}$
be the maximal radius of balls centred at $\overline{\x}_K$ and included in $K$.
Fixing $\eta>0$ such that
\be\label{reg:mesh}
\forall K\in\mesh\,,\; \eta\ge \frac{{\rm diam}(K)}{\rho_K}\,,
\ee
we use the following extension of the notation \eqref{notation:lesssim}:
\[
\begin{aligned}
&X\lesssim_{\eta} Y\mbox{ means that }X\le CY\mbox{ for some $C$ depending}\\
&\mbox{only on $\O$, $A$, an upper bound of $C_\disc$, and $\eta$}.
\end{aligned}
\]
\begin{theorem}[Super-convergence for post-processed controls I] \label{thm.super-convergence}
	Let $\disc$ be a GD and $\mesh$ be a mesh.
	Assume that 
	\begin{itemize}
		\item \assum{1}--\assum{4} hold,
		\item  $\by$ and $\bp$ belong to $H^2(\O)$,
		\item $\yd$ and $f$ belong to $H^1(\O)$,
	\end{itemize}
	and let $\tu$, $\tu_h$ be the post-processed controls defined by \eqref{Projection_ppu} where $\bp$ and $\bp_\disc$ are chosen such that $\dashint_{\O}^{}\Pi_\disc \bp_\disc \d\x =\dashint_{\O}^{} \bp\d\x$.
	Then there exists $C$ depending only on $\alpha$ in \eqref{def:cost} such that
	\be\label{eq:supercv}
	\norm{\tu-\tu_{h}}{} \lesssim_{\eta}Ch^{2-\frac{1}{2^*}}\norm{\bu}{W^{1,\infty}(\mesh_1)}
	+Ch^2\mathcal F(\yd,f,\by,\bp),
	\ee
	\begin{align*}
	\mathcal F(\yd,f,\by,\bp)={}\norm{\yd}{H^1(\O)}+\norm{f}{H^1(\O)}+\norm{\by}{H^2(\O)}+\norm{\bp}{H^2(\O)}.
	\end{align*}
\end{theorem}
\begin{theorem}[Super-convergence for post-processed controls II] \label{thm.fullsuper-convergence}
	Let the assumptions and notations of Theorem \ref{thm.super-convergence} hold, except \assum{3} which is replaced
	by: 
	\be\label{Linfty.est}
	\begin{aligned}
		&\mbox{there exists $\delta>0$ such that, for any $F\in L^2(\O)$,}\\
		&\mbox{the solution $\psi_D$ to \eqref{base.GS} satisfies $\norm{\Pi_D \psi_D}{L^\infty(\O)}
			\le \delta \norm{F}{}$.}
	\end{aligned}
	\ee
	Then there exists $C$ depending only on $\alpha$ and $\delta$ such that
	\be \label{eq:fullsupercv}
	\norm{\tu-\tu_h}{}\lesssim_{\eta} C h^2\left[\norm{\bu}{W^{1,\infty}(\mesh_1)}+\mathcal F(\yd,f,\by,\bp)\right].
	\ee
\end{theorem}
\begin{remark}
For most methods, assumption \eqref{Linfty.est} is satisfied if the mesh is quasi-uniform (see  \cite{LGRN_NCFEM} for conforming and non-conforming $\mathbb{P}_1$
finite element method, and \cite[Theorem 7.1]{GDM_control} for HMM methods).
\end{remark}

\begin{corollary}[Super-convergence for the state and adjoint variables] \label{cor.super-convergence}
	Under the assumptions of Theorem \ref{thm.super-convergence}, the following error estimates hold, with $C$ depending only on $\alpha$:
	\begin{align}
	\label{eq_supercv.y}
	\norm{\by_\mesh-\Pi_\disc \by_\disc}{}\lesssim_{\eta}{}&
	Ch^{r}\norm{\bu}{W^{1,\infty}(\mesh_1)}+Ch^2 \mathcal F(\yd,f,\by,\bp),\\
	\label{eq_supercv.p}
	\norm{\bp_\mesh-\Pi_\disc \bp_\disc}{}\lesssim_{\eta}{}&
	C h^{r}\norm{\bu}{W^{1,\infty}(\mesh_1)}
	+Ch^2\mathcal F(\yd,f,\by,\bp),
	\end{align}
	where $\by_\mesh$ and $\bp_\mesh$ are defined as in \assum{1}, and
	$r=2-\frac{1}{2^*}$.
	
	Under the assumptions of Theorem \ref{thm.fullsuper-convergence}, \eqref{eq_supercv.y} and \eqref{eq_supercv.p}
	hold with $r=2$ and $C$ depending only $\alpha$ and $\delta$.
\end{corollary}

 \subsection{Discussion on post-processed controls}\label{dis:ppu}
 In this section, we present a detailed analysis of post-processed controls given by \eqref{Projection_ppu}.
This analysis is performed under the assumptions of Section \ref{supercv}, and by also assuming
that $\WS_\disc(\varphi)\lesssim h$ for all $\varphi\in H^2(\O)$ (see Remark \ref{rate.ellptic}). We begin by stating and proving two lemmas which discuss projection relations between control and adjoint variables for the pure Neumann problem, both at the continuous level and at the discrete level. We then show that the post-processed controls remain $O(h)$ close to their corresponding original controls, see \eqref{est:uh:ppuh} and \eqref{est:u:ppu}. Hence, the super-convergence result makes sense: since $\bu_h$ is piecewise constant, it is impossible to expect more than $O(h)$ approximation on the controls; but by \enquote{moving} these controls by a specific $O(h)$, we obtain computable post-processed controls that enjoy an $O(h^2)$ convergence result.
 
 \begin{lemma} \label{cont_proj_general} 
 	Let $-\infty\le a<0<b\le \infty$ and $\phi \in L^1(\O)$.
 	Define $\Gamma:\R \rightarrow \R$ by 
\[
\Gamma(c)=\int_{\O}P_{[a,b]}(\phi+c)\d\x,
\]
where $P_{[a,b]}$ is given by \eqref{def_proj[a,b]}. 
Set $m=a-\mbox{\rm ess sup}(\phi)\in [-\infty,+\infty)$ and $M=b-\mbox{\rm ess inf}(\phi)\in (-\infty,+\infty]$.
Then we have the following.
 	\begin{enumerate}
 		\item\label{it1} $\Gamma$ is Lipschitz continuous.
 		\item\label{it2} $\lim\limits_{c \to m}\Gamma(c)=a|\O|$, $\lim\limits_{c \to M}\Gamma(c)=b|\O|$, and there
is $c^\star\in (m,M)$ such that $\Gamma(c^\star)=0$.
 		\item\label{it3} If $\phi\in H^1(\O)$, then for any compact interval $Q$ in $(m,M)$, there exists $\rho_Q >0$ such that if $c, c' \in Q$ with $c < c'$, then 
\begin{equation}\label{gamma.incr}
\Gamma(c')-\Gamma(c) \ge \rho_Q(c'-c).
\end{equation}
 		As a consequence, the real number $c^\star$ in Item \ref{it2} is unique.
 	\end{enumerate}
 \end{lemma}

 \begin{proof}
 	Item \ref{it1} is obvious since $P_{[a,b]}$ is Lipschitz continuous.
 	
 	\medskip
 	
 	Let us now analyse the limits in Item \ref{it2}. Let $(c_n)$ be a sequence in $\R$ such that $c_n \rightarrow M$ as $n \rightarrow \infty$. By definition of $M$, this implies $P_{[a,b]}(\phi+c_n) \rightarrow b$ a.e on $\O$. Let $(c_n)$ be bounded below by $R$ and note that $\phi +R \in L^1(\O)$. Moreover, for $s \in \R$, $a<0<b$ implies $P_{[a,b]}(s)\ge \min(s,0)$ so $P_{[a,b]}(\phi+c_n)\ge \min(\phi+c_n,0)\ge \min(\phi+R,0)\in L^1(\O)$. By Fatou's Lemma, 
$$\int_{\O}^{}b\d\x \le \liminf_{n \rightarrow \infty} \int_{\O}^{}P_{[a,b]}(\phi(\x)+c_n)\d\x$$
 	which gives $b|\O| \le \liminf_{n \rightarrow \infty}\Gamma(c_n)$. Since $\Gamma(c_n) \le b|\O|$ (because
$P_{[a,b]}(s)\le b$), we infer
 	that $\lim_{n\to\infty}\Gamma(c_n)=b|\O|$, and thus that $\lim_{c\to M}\Gamma(c)=b|\O|$.
 	In a similar way, we deduce that $\lim_{c \to m}\Gamma(c)=a|\O|$.
 	
 	The existence of $c^\star$ such that $\Gamma(c^\star)=0$ then follows from the intermediate value theorem
 	and $\lim_{c \to m}\Gamma(c)=a|\O|<0< b|\O|=\lim_{c \to M}\Gamma(c)$.
 	
 	\medskip
 	
 	We now assume that $\phi\in H^1(\O)$ and we turn to Item \ref{it3}. For a.e $c \in \R$, $\Gamma^{'}(c)=\int_{\O}\mathbbm{1}_{(a,b)}(\phi(\x)+c)\d\x,$  where $\mathbbm{1}_{(a,b)}$ is the characteristic function of $(a,b)$.
 	Define $\Theta(c)=\int_{\O}\mathbbm{1}_{(a,b)}(\phi(\x)+c)\d\x,$ for all $c \in \R$. We claim that
 	\begin{itemize}
 		\item  $\Theta$ is lower semi-continuous,
 		\item $\forall c \in (m,M)$, $\Theta(c) >0.$
 	\end{itemize}
 	To prove that $\Theta$ is lower semi-continuous, let $c_n \rightarrow c$ as $n \rightarrow \infty$. 
Since $\mathbbm{1}_{(a,b)}$ is lower semi-continuous on $\R$, we have, for all $\x\in\O$,
 	$$\mathbbm{1}_{(a,b)}(\phi(\x)+c)\le \liminf_{n \rightarrow \infty}\mathbbm{1}_{(a,b)}(\phi(\x)+c_n).$$
 	Applying Fatou's Lemma,
 	\[
 	\Theta(c) \le \liminf_{n \rightarrow \infty}\int_{\O}^{}\mathbbm{1}_{(a,b)}(\phi(\x)+c_n)\d\x= \liminf_{n \rightarrow \infty} \Theta(c_n).
 	\]
 	Hence, $\Theta$ is lower semi-continuous.
 	We now show that  $\Theta >0$ on $(m,M).$ Let $c \in (m,M)$. Then $I=(a-c,b-c)\cap(\mbox{ess inf }\phi,\mbox{ess sup }\phi)$ is an interval of positive length, since $a-c < \mbox{ess sup }\phi$ and $b-c >\mbox{ess inf }\phi$. The set $W_{I,c}=\{\x:\phi(\x) \in I\}$ has a non-zero measure because $\phi \in H^1(\O)$ and $\O$ is connected. To see this,
 	let $\alpha<\beta$ be the endpoints of $I$ and assume that $\phi \in H^1(\O)$ takes some values less than $\alpha$ on a non-null set, some values greater than $\beta$ on a non-null set, but that $W_{I,c}$ is a null set. Then $P_{[\alpha,\beta]}(\phi) \in H^1(\O)$ exactly takes the values $\alpha$ and $\beta$ (outside a set of zero measure). Hence $\nabla P_{[\alpha,\beta]}(\phi)=\mathbbm{1}_{[\alpha,\beta]}(\phi)\nabla \phi=0$ and
 	$P_{[\alpha,\beta]}(\phi)$ should be constant, since $\O$ is connected, which is a contradiction. Thus, $W_{I,c}$ has a non-zero measure. Since $\{\x:\phi(\x)+c \in (a,b)\} \supseteq W_{I,c}$, this gives $\Theta(c) \ge |W_{I,c}|>0$.
 	
 	Coming back to Item \ref{it3}, let $Q$ be a compact interval in $(m,M)$. We know that $\Theta>0$ on $Q$ and $\Theta$ is lower semi-continuous. Hence $\Theta$ reaches its minimum on $Q$ and $\inf_{Q} \Theta=\Theta(c_0)>0$ for some $c_0 \in Q$. Since $\Gamma^{'}=\Theta$ a.e, $\Gamma^{'}\ge \inf_{Q}\Theta$ a.e on $Q$ and, $\Gamma$ being Lipschitz and $[c,c']\subset Q$, we infer
 	$$\Gamma(c')-\Gamma(c) =\int_{c}^{c'}\Gamma^{'}(s)\d s \ge \left[\inf_{Q}\Theta\right](c'-c),$$
 	which establishes \eqref{gamma.incr}. The uniqueness of $c^\star$ such that $\Gamma(c^\star)=0$ follows from 
 	this inequality, which shows that $\Gamma$ is strictly increasing on $(m,M)$.
 \end{proof}
 
 \begin{lemma}[Projection formulas for the controls]\label{u_proj_c}
 	If $\bp \in H^1(\O)$ is a co-state and $\oc \in \R$ is such that $\dashint_{\O}^{}P_{[a,b]}(-\frac{1}{\alpha}\bp(\x)+\oc)\d\x=0$, then the continuous optimal control $\bu$ in \eqref{continuous_kkt} can be expressed in terms of the projection formula 	
 	\be \label{proj_cont}
 		\bu(\x)=P_{[a,b]}\left(-\frac{1}{\alpha}\bp(\x)+\oc\right).
 		\ee
 	If $\disc$ is a GD and $\bp_\disc$ is chosen such that 
 	\begin{equation}\label{proper.pD}
	\dashint_{\O}P_{[a,b]}\left(\Proj\left(-\frac{1}{\alpha}\Pi_\disc \bp_\disc \right) \right)\d\x = 0,
	\end{equation}
	then the discrete optimal control in \eqref{discrete_kkt} is given by
 	\be\label{proj_GDM}
 	\bu_{h}(\x)=P_{[a,b]}\left(\Proj\left(-\frac{1}{\alpha}\Pi_\disc \bp_\disc(\x) \right) \right).
 	\ee
 \end{lemma}
 \begin{proof}
 	Set $\tp=\bp-\alpha \oc$. Clearly, $\tp \in H^1(\O)$.
 	From the optimality condition for the control problem \eqref{opt_cont}, we deduce that
 	$$ (\tp+\alpha\bu,v-\bu)\geq 0 \: \; \forall v\in  \Uad,$$
 	since $\int_{\O}^{}\bu \d\x=\int_{\O}^{}v \d\x=0$. Set $U=P_{[a,b]}(-\alpha^{-1}\tp)$ i.e,
 	\[U = \begin{cases} 
 	a &  \mbox{ on }\O_+=\{\x \in \O\,:\,\tp(\x)+\alpha U(\x)>0\} \\
 	-\alpha^{-1}\tp&\mbox{ on }\O_{0}=\{\x\in\O\,:\,\tp(\x)+\alpha U(\x)=0\}\\
 	b &  \mbox{ on }\O_-=\{\x \in\O\,:\,\tp(\x)+\alpha U(\x)<0\}.
 	\end{cases}
 	\]
 	It is then straightforward to see that $U \in \Uad$, i.e, $U \in [a,b]$ and $\dashint_{\O}U(\x)d\x=0$ (by choice of $\oc$). 
 	Then, using the definitions of $\O_+$, $\O_0$ and $\O_-$, since $v\ge a=U$ on $\O_+$ and $v\le b=U$ on $\O_-$,
 	\begin{align*}
 		(\tp+\alpha U,v-U) &=\int_{\O_{+}}(\tp+\alpha U)(v-U)\d\x+\int_{\O_{0}}(\tp+\alpha U)(v-U)\d\x\\
 		&\qquad +\int_{\O_{-}}(\tp+\alpha U)(v-U)\d\x\ge 0.
 	\end{align*}
 	Recall that the optimality condition is nothing but a characterisation of the $L^2(\O)$ orthogonal projection of $-\alpha^{-1}\tp$ on $\Uad$ and, as such, defines a unique element $\bu$ of $\Uad$.
 	We just proved that $U=P_{[a,b]}(-\alpha^{-1}\tp)$ satisfies this optimality condition, which shows that it
 	is equal to $\bu$. The proof of \eqref{proj_cont} is complete.
 	
 	The second relation follows in a similar way by noticing that, since controls are piecewise-constants on $\mesh$, \eqref{opt_discrete} is equivalent to
 	$ (\Proj(\Pi_\disc \bp_\disc) +\alpha\bu_h,v_h-\bu_h)\geq 0$, for all $v_h \in  \Uadh$. We also notice that, by definition of $\Proj$ and the assumption \eqref{proper.pD}, $P_{[a,b]}\left(\Proj\left(-\frac{1}{\alpha}\Pi_\disc \bp_\disc \right)\right) \in \Uadh.$
 \end{proof}

  \begin{remark} \label{remark:proj.relation}
There is at least one adjoint $\bp_\disc$ such that \eqref{proper.pD} is satisfied: start from any adjoint $\bp^0_\disc$ and, by applying Lemma \ref{cont_proj_general}
(Item \ref{it2}) to $\phi=\Proj(-\alpha^{-1}\Pi_\disc\bp_\disc^0)$ and by noticing that $\phi+c^\star=\Proj(-\alpha^{-1}\Pi_\disc\bp_\disc^0)+c^\star$,
find $c^\star$ such that $\bp_\disc=\bp_\disc^0-\alpha c^\star1_\disc$ satisfies \eqref{proper.pD}.

Since the discrete co-state $\bp_\disc$ is a computable quantity, its average is easier to
fix than the average of the non-computable $\bp$. Hence, the projection relation \eqref{proj_GDM} is the most natural choice to express the discrete control $\bu_h$ in terms of the discrete adjoint variable. This is the choice made in the modified active set strategy presented in Subsection \ref{activeset}. Once this choice is made,  since $\bp$ must have the same average as $\Pi_\disc \bp_\disc$ for $\tu$ defined in \eqref{Projection_ppu} to satisfy super-convergence
estimates, it is clear that $P_{[a,b]}(-\frac{1}{\alpha}\bp)$ will not have a zero average in general. Hence, if we want to express the continuous control in terms of $\bp$, we need to offset this $\bp$ by the correct $\overline{c}$, as stated in the lemma. 
  \end{remark}

\begin{lemma}[Stability of the discrete states]\label{discrete.bdd}
Let $\disc$ be a GD, $(\by, \bp, \bu)$ be a solution to \eqref{continuous_kkt} and
$(\by_\disc, \bp_\disc, \bu_h)$ be a solution to \eqref{discrete_kkt}. Assume that $\dashint_{\O}^{}\Pi_\disc \bp_\disc \d\x =\dashint_{\O}^{} \bp\d\x$. Then 
\begin{align}
\norm{\Pi_\disc \by_\disc}{} +\norm{\nabla_\disc \by_\disc}{} + \norm{\Pi_\disc \bp_\disc}{}+\norm{\nabla_\disc \bp_\disc}{}& \lesssim 1. \label{discrete.state.adj.bdd}
\end{align}
\end{lemma}
\begin{proof}
	Let $\phi \in H^{\div}_0(\O)$. Taking $w=0$ in \eqref{def.SD}, we get
	$S_\disc(\phi)\le \norm{\phi}{}+\norm{\nabla \phi}{}.$ By Cauchy-Schwarz inequality, using \eqref{def.CD} and \eqref{def:norm}, for $w \in X_\disc$, 
	\begin{align*}
		\int_\O (\Pi_\disc  w \: \div(\phi)+\nabla_\disc w\cdot\phi)\d\x &\le \norm{\Pi_\disc w}{}\norm{\div(\phi)}{}+\norm{\nabla_\disc w}{}\norm{\phi}{}\\
		&\le C_\disc \norm{w}{\disc}\left(\norm{\div(\phi)}{}+\norm{\phi}{}\right).
	\end{align*}
	With \eqref{def.WD}, this implies $\WS_\disc(\phi) \le C_\disc \left(\norm{\div(\phi)}{}+\norm{\phi}{}\right)$.
	\smallskip

	Therefore, for $A\nabla\psi \in H^{\div}_0(\O)$, recalling the definition \eqref{def.ws} of $\WS_\disc$, we have
	$$\WS_\disc(\psi)\lesssim \norm{\psi}{H^1(\O)}+\norm{\div(A\nabla\psi)}{}+\norm{A\nabla\psi}{} \lesssim \norm{\psi}{H^1(\O)}+\norm{\div(A\nabla\psi)}{} \lesssim 1. $$
	Using Proposition \ref{prop.state.adj}, Theorem \ref{theorem.control}, this shows that 
	$$\norm{\nabla_\disc \by_\disc}{}+ \norm{\Pi_\disc \by_\disc}{}\lesssim\norm{\Pi_\disc \by_\disc-\by}{}+ \norm{\nabla_\disc \by_\disc-\nabla\by}{}  \lesssim 1.$$
	The result for the adjoint variable can be derived similarly  and hence \eqref{discrete.state.adj.bdd} follows.
\end{proof}

In the rest of this section, we establish $\mathcal O(h)$ estimates
between the controls $\bu$, $\bu_h$ and their post-processed versions $\tu$, $\tu_h$.
These estimates justify that the post-processed controls are indeed meaningful quantities.

\medskip

Using \eqref{Projection_ppu}, \eqref{proj_GDM}, the Lipschitz continuity of $P_{[a,b]}$, $\assum{2}$ and Lemma \ref{discrete.bdd}, we have the following estimate between $\bu_h$ and $\tu_h$:
\be\label{est:uh:ppuh}
\norm{\tu_h-\bu_h}{}\le \alpha^{-1}\norm{\Pi_\disc \bp_\disc-\Proj\left(\Pi_\disc \bp_\disc \right)}{}\lesssim\alpha^{-1}h\norm{\nabla_\disc \bp_\disc}{}\lesssim \alpha^{-1}h.
\ee

\medskip

Let us now turn to estimating $\bu-\tu$.
 The co-state $\bp \in H^1(\O)$ in \eqref{continuous_kkt} is still taken such that $\dashint_{\O}^{} \bp\d\x=\dashint_{\O}^{}\Pi_\disc \bp_\disc \d\x$. From Lemma \ref{cont_proj_general}, it follows that there exists a unique constant $\oc \in (m,M)$ such that $\int_{\O}^{}P_{[a,b]}(-\frac{1}{\alpha}\bp+\oc)\d\x=0$, where $m$ and $M$ are defined as in Lemma \ref{cont_proj_general}. Using Lemma \ref{u_proj_c} and recalling \eqref{Projection_ppu}, 
 \be\label{cont_u_ppu}
 \bu(\x)=P_{[a,b]}\left(-\frac{1}{\alpha}\bp(\x)+\oc\right)\quad\mbox{ and }\quad   \tu(\x)={}P_{[a,b]}\left(-\frac{1}{\alpha}\bp_\mesh(\x) \right).
 \ee
 Starting from \eqref{cont_u_ppu} and using the Lipschitz continuity of $ P_{[a,b]}$, the assumption $\WS_\disc(\varphi)\lesssim h$, Corollary \ref{cor.super-convergence} and Proposition \ref{prop.state.adj}, we get a constant  $C$ depending only on  $\alpha$, $f$, $\by_d$, $\bp,$ $\by$ and $\bu$  such that
 \begin{align}
 \norm{\bu-\tu}{}\le {\alpha}^{-1}\norm{\bp_\mesh-\bp+\alpha \oc}{} 
 &\lesssim {\alpha}^{-1}\norm{\bp_\mesh-\bp}{} +|\oc|\nonumber \\
 &\lesssim {\alpha}^{-1}\left(\norm{\bp_\mesh-\Pi_\disc\bp_\disc}{} +\norm{\Pi_\disc\bp_\disc-\bp}{}\right)+|\oc|\nonumber \\
& \lesssim_\eta Ch+|\oc|. \label{c_est}
 \end{align}

 To estimate the last term in \eqref{c_est}, recall the definition of $\Gamma(c)$ from Lemma \ref{cont_proj_general} for $\phi=\alpha^{-1}\bp$.
 $$\Gamma(c)=\int_{\O}^{}P_{[a,b]}\left(-\frac{1}{\alpha}\bp+c\right)\d\x.$$
By choice of $\oc$, $\Gamma(\oc)=0$. From Lemma \ref{u_proj_c}, the choice of $\bp_\disc$ shows that
 \begin{align*}
	\Gamma(0)={}&\int_{\O}^{}P_{[a,b]}\left(-\frac{1}{\alpha}\bp\right)\d\x\\
	={}&\int_{\O}^{}\left(P_{[a,b]}\left(-\frac{1}{\alpha}\bp\right)-P_{[a,b]}\left(\Proj\left(-\frac{1}{\alpha}\Pi_\disc \bp_\disc \right) \right)\right)\d\x.
	\end{align*}
Let $q_\disc$ be the solution to \eqref{discrete_adjoint} with source term $\by-\yd$ (that is, the
solution to the GS for the equation \eqref{adj_cont} satisfied by $\bp$) such that $\dashint_\O\Pi_\disc q_\disc \d\x=\dashint_\O\Pi_\disc \bp_\disc \d\x=\dashint_\O\bp\d\x$. Using the Lipschitz continuity of $ P_{[a,b]}$, the Cauchy--Schwarz inequality, the triangle inequality, Remark \ref{avg_err_estimate}, Proposition \ref{prop.stab}, $\assum{2}$, Theorem \ref{th:error.est.PDE} and Lemma \ref{discrete.bdd}, we obtain
 \begin{align}
 |\Gamma(0)| &\le \int_{\O}^{}\bigg|P_{[a,b]}\left(-\frac{1}{\alpha}\bp\right)-P_{[a,b]}\left(\Proj\left(-\frac{1}{\alpha}\Pi_\disc \bp_\disc \right) \right)\bigg|\d\x \nonumber\\
 &\lesssim {\alpha}^{-1}\norm{\bp-\Proj(\Pi_\disc \bp_\disc)}{}\nonumber\\
 &\lesssim {\alpha}^{-1}\norm{\bp-\Pi_\disc q_\disc}{}+{\alpha}^{-1}\norm{\Pi_\disc q_\disc-\Pi_\disc \bp_\disc}{}+ {\alpha}^{-1}\norm{\Pi_\disc \bp_\disc-\Proj\left(\Pi_\disc \bp_\disc\right)}{} \nonumber \\
  &\lesssim {\alpha}^{-1} \left(\WS_\disc(\bp)+{\alpha}^{-1}\norm{\by-\Pi_\disc \by_\disc}{}+h\norm{\nabla_\disc \bp_\disc}{} \right) \nonumber \\
&\lesssim \alpha^{-1}\left(\WS_\disc(\bp)+ \WS_\disc(\by)+ h\norm{\nabla_\disc \bp_\disc}{}\right)\lesssim \alpha^{-1}h. \label{Gamma.0}
 \end{align}
Let $m,M$ be as in Lemma \ref{cont_proj_general} for $\phi=\alpha^{-1}\bp$. 
Relation \eqref{Gamma.0} shows that $a|\O|<\Gamma(0)<b|\O|$ if $h$ is small
enough; hence, in this case, $0\in (m,M)$. There is therefore a compact interval $Q$ in $(m,M)$ depending only on $\bp$ such that 0 and $\oc$ belong to $Q$. Without loss of generality, we can assume that $\oc\ge 0$. A use of Lemma \ref{cont_proj_general} leads to
 $$\Gamma(\oc)-\Gamma(0)\ge \rho_Q \oc,$$
 where $\rho_Q >0$. This implies $0 \le \oc \lesssim \alpha^{-1}h/\rho_Q$, using \eqref{Gamma.0} and the fact that $\Gamma(\oc)=0$. Combining this with \eqref{c_est}, we infer that 
 \be\label{est:u:ppu}
 \norm{\tu-\bu}{} \lesssim_\eta \left(C+\frac{\alpha^{-1}}{\rho_Q}\right)h.
 \ee

 \section{Numerical experiments} \label{sec:numericalexp}
In this section, we first present the modified active set strategy. This is followed by results of numerical experiments for conforming, non-conforming and mimetic finite difference methods.
 \subsection{A modified active set strategy}\label{activeset}
The interest of choosing an adjoint given by \eqref{proper.pD} is highlighted in Lemma \ref{u_proj_c}: we have the projection relation \eqref{proj_GDM} between the discrete control and adjoint.
Such a relation is at the core of the (standard) active set algorithm \cite{tf}. For a detailed analysis of this method, we refer the reader to \cite{MBIKKK, MBKK, KR02}. Here, we propose a modified active set algorithm that enforces the proper zero average condition,
and thus the proper relation between discrete adjoint and control.

We first notice that, when selecting the $\bp_\disc$ such that \eqref{proper.pD} holds, the KKT optimality conditions \eqref{discrete_kkt} can be rewritten as:
Seek $(\by_{\disc}, \bp_{\disc}, \bu_{h})\in X_{\disc} \times X_{\disc} \times \Uadh$, such that
 \begin{subequations} \label{discrete_kkt1}
 	\begin{align}
 	& a_{\disc}(\by_{\disc},w_{\disc})+\rho\left(\dashint_{\O}\Pi_\disc \by_\disc\d\x\right)\left(\dashint_{\O}\Pi_\disc w_\disc\d\x\right)  \nonumber\\
	&\qquad \qquad \qquad \qquad \qquad  \qquad \qquad = ( \bu_{h} + f, \Pi_\disc w_{\disc})  \qquad\forall \: w_{\disc} \in  X_{\disc},\label{state_discrete1}\\
 	&a_\disc(z_\disc,\bp_\disc)+\rho\left(\dashint_{\O}P_{[a,b]}\left[\Proj(-\alpha^{-1}\Pi_\disc \bp_\disc)\right] \d\x\right)\left(\dashint_{\O} \Pi_\disc z_\disc\d\x\right) \nonumber \\
 	& \qquad \qquad \qquad \qquad \qquad \qquad  \qquad  =( \Pi_\disc \by_\disc-\yd,\Pi_\disc z_\disc) \qquad \forall \: z_{\disc} \in  X_{\disc},\label{adj_discrete1} \\
 	&\qquad \qquad \qquad \qquad (\Pi_\disc \bp_{\disc} +\alpha \bu_h,v_h-\bu_h)\geq  0
 	\qquad  \forall \: v_h \in  \Uadh, \label{opt_discrete1}
 	\end{align}	
 \end{subequations} 
 where $\rho> 0$ is constant. 

 Set $\bar{\mu}_h=-(\alpha^{-1}\Pi_\disc \bp_\disc+\bu_h)$. As the original active set strategy \cite{tf}, the modified active set strategy is an iterative algorithm. As initial guesses, two arbitrary functions, $u_h^0, \mu_h^0$ are chosen. In the $n$th step of the algorithm, we define the set of active and inactive restrictions by
 $$A^{n}_{a,h}(\x)=\{\x: u_h^{n-1}(\x)+\mu_h^{n-1}(\x)<a\},\,\, A^{n}_{b,h}(\x)=\{\x: u^{n-1}_h(\x)+\mu^{n-1}_h(\x)>b\}, $$
 $$I^n_h=\O\setminus (A^{n}_{a,h} \cup A^{n}_{b,h}).$$ 
If $\max\left(\frac{\norm{u_h^{n}-u_h^{n-1}}{L^\infty(\O)}}{\norm{u_h^{n-1}}{L^\infty(\O)}},\frac{\norm{\Pi_\disc p_\disc^n-\Pi_\disc p_\disc^{n-1}}{L^\infty(\O)}}{\norm{\Pi_\disc p_\disc^{n-1}}{L^\infty(\O)}}\right)\le 10^{-10}$, then we terminate the algorithm. In this case, we notice that the relative $L^\infty$ difference between $\Pi_\disc y_\disc^{n-1}$ and $\Pi_\disc y_\disc^n$ is less than $10^{-6}$ for all examples in Section \ref{examples}. Else we
  find $y_\disc^n, {p}_\disc^n$ and $u_h^n$ solution to the system
\begin{subequations}\label{active.original}
  \begin{align}
  	a_\disc(y_\disc^n,w_\disc)+\rho&\left(\dashint_{\O}\Pi_\disc y_\disc^n\d\x\right)\left(\dashint_{\O}\Pi_\disc w_\disc\d\x\right) \nonumber\\
&\qquad\qquad= ( u_h^n+f, \Pi_\disc w_\disc)\,\qquad\forall w_\disc\in X_\disc,\\
  	a_\disc(z_\disc,p_\disc^n)+\rho&\left(\dashint_{\O}P_{[a,b]}\left[\Proj(-\alpha^{-1}\Pi_\disc p_\disc^n)\right] \d\x\right)\left(\dashint_{\O} \Pi_\disc z_\disc\d\x\right)\nonumber\\
  	&\qquad  \qquad  = ( \Pi_\disc y_\disc^n-\yd,\Pi_\disc z_\disc)\,\qquad\forall z_\disc\in X_\disc, 
  	  \end{align}
  	 \begin{equation}
			u_h^n = \begin{cases} 
  	 a &  \mbox{ on }A^{n}_{a,h} \\
  	 \Proj(-\alpha^{-1}\Pi_\disc p_\disc^n)&  \mbox{ on } I^n_h\\
  	 b & \mbox{ on }  A^{n}_{b,h}.
  	 \end{cases}
  	 \end{equation}
\end{subequations}

 The above algorithm consists of non-linear equations. It can however be approximated by a linearized system in the following way, thus leading to our final modified active set algorithm. Instead of solving \eqref{active.original}, we solve
 \begin{subequations} \label{linear_algo_discrete}
 	\begin{align}
 	a_\disc(y_\disc^n,w_\disc)+\rho&\left(\dashint_{\O}\Pi_\disc y_\disc^n\d\x\right)\left(\dashint_{\O}\Pi_\disc w_\disc\d\x\right) \nonumber\\
&\qquad\qquad= ( u_h^n+f, \Pi_\disc w_\disc)\,\qquad\forall w_\disc\in X_\disc, \label{state_algo_discrete} \\
 	a_\disc(z_\disc,p_\disc^n)+\rho&\left(\dashint_{\O} \Pi_\disc p_\disc^n\d\x\right)\left(\dashint_{\O} \Pi_\disc z_\disc\d\x\right)\nonumber\\
 &\qquad \qquad = ( \Pi_\disc y_\disc^n-\yd,\Pi_\disc z_\disc)+\rho S_\disc^{n-1}\,\qquad\forall z_\disc\in X_\disc, \label{adjoint_algo_discrete}
 	 	\end{align}
 \be
  u_h^n = \begin{cases}
  a &  \mbox{ on }A^{n}_{a,h} \\
  \Proj(-\alpha^{-1}\Pi_\disc p_\disc^n)&  \mbox{ on } I^n_h  \\
  b & \mbox{ on }  A^{n}_{b,h},
  \end{cases}
   \label{control_algo_discrete}
  \ee
 \end{subequations}
 where 
\[
	S_\disc^{n-1}=\left(\dashint_{\O}\Pi_\disc z_\disc\d\x\right)\left(\dashint_{\O}\left\{\Pi_\disc p_\disc^{n-1}-P_{[a,b]}\left[\Proj(-\alpha^{-1}\Pi_\disc p_\disc^{n-1})\right]\right\}\d\x\right).
\]
Note that \eqref{control_algo_discrete} can be re-written in the following more commonly used form:
$$u_h^n + \left(1-\mathbbm{1}_{a,h}^n-\mathbbm{1}_{b,h}^n\right)\alpha^{-1}\Pi_\disc p_\disc^n=\mathbbm{1}_{a,h}^na+\mathbbm{1}_{b,h}^nb,$$
where $\mathbbm{1}_{a,h}^n$ and $\mathbbm{1}_{b,h}^n$ denote the characteristic functions of the sets $A^{n}_{a,h}$ and $A^{n}_{b,h}$ respecively.

 \begin{remark}
The convergence analysis of the proposed algorithm is a plan for future study. However, if $(\Pi_\disc y_\disc^n, \Pi_\disc p_\disc^n)$  converges weakly to $(\Pi_\disc \by_\disc,\Pi_\disc \bp_\disc)$ in $H^1(\O)$ and $u_h^n$ converges to $\bu_h$ in $L^2(\O)$, then the solution to \eqref{linear_algo_discrete} converges to the solution of \eqref{discrete_kkt1}  as $n \rightarrow \infty$.
 \end{remark}
 
\subsection{Examples} \label{examples}
In this section, we illustrate examples for the numerical solution of \eqref{model.neumann}.
We use three specific schemes for the state and adjoint variables: conforming finite element (FE) method, non-conforming finite element (nc$\mathbb{P}_1$FE) method and hybrid mimetic mixed (HMM) method. All three are GDMs
with gradient discretisations with bounds on $C_\disc$, order $h$ estimate on $\WS_\disc$, and
satisfying assumptions \assum{1}--\assum{4}, and \eqref{Linfty.est} on quasi-uniform meshes;
see \cite{gdm,GDM_control} and Remarks \ref{rate.ellptic} and \ref{remark:supercv}.

The control variable is discretised using piecewise constant functions on the corresponding meshes. The discrete solution is computed using the modified active set algorithm mentioned in Subsection \ref{activeset} with zero as an initial guess for both $u$ and $\mu$. Here, $U_{a}$ and $Y_{a}$ denote the average values of the computed control $\bu_h$ and the reconstructed state solution $\Pi_\disc \by_\disc$ respectively. Let $\texttt{ni}$ denote the number of iterations required for the convergence of the modified active set algorithm, and $f_{a}$ denote the numerical average of the source term $f$ calculated using the same quadrature rule as in the implementation of the schemes, i.e, $$f_a=\frac{1}{|\O|}\sum_{K \in \mesh}|K|f(\overline{\x}_K),$$ where $\overline{\x}_K$ denotes the center of mass of the cell $K$. This numerical average enables us to evaluate the quality of the quadrature rule for each mesh; in particular, since $f$ has a zero average, any quantity of the order of $f_a$ can be
considered to be equal to zero, up to quadrature error. The relative errors are denoted by
\[
\err_\disc(\by):=\frac{\norm{\Pi_\disc \by_\disc -\by_\mesh}{}}{\norm{\by}{}},\quad
\err_\disc(\nabla\by) :=\frac{\norm{\nabla_\disc \by_\disc -\nabla\by}{}}{\norm{\nabla\by}{}},
\]
\[
\err_\disc(\bp):=\frac{\norm{\Pi_\disc \bp_\disc -\bp_\mesh}{}}{\norm{\bp}{}},\quad
\err_\disc(\nabla\bp) :=\frac{\norm{\nabla_\disc \bp_\disc -\nabla\bp}{}}{\norm{\nabla\bp}{}},
\]
\[
\err(\bu):=\frac{\norm{\bu_h -\bu}{}}{\norm{\bu}{}},
\quad\mbox{ and }\quad
\err(\tu) :=\frac{\norm{\tu_h - \tu}{}}{\norm{\bu}{}}.
\]

The data in the optimal control problem \eqref{model.neumann} are chosen as follows:
\begin{align*}
&\by=2\cos(\pi x)\cos(\pi y), \quad \bp=2\cos(\pi x)\cos(\pi y),\\
&   \quad \alpha=1, \quad \Uad=[a,b],\quad \bu= P_{[a,b]}\left( - \bp+\oc\right),
\end{align*}
where $\oc$ is chosen to ensure that $\dashint_\O \bu\d\x=0$.
The matrix-valued function is given by $A={\rm Id}$ unless otherwise specified. The source term $f$ and the desired state  $\yd$ are then computed using
\begin{equation*}
f=-\Delta \by-\bu, \quad \yd=\by+\Delta \bp.
\end{equation*}


	\subsubsection{\textbf{Example 1 :}}$\O$ = $\left(0,1\right)^2$, $\rho = 10^{-4}$, $a=-1$, $b=1$.

We here consider the computational domain $\O$ = $\left(0,1\right)^2$. We have $\bp(x,y)=-\bp(1-x,y)$ and, since $P_{[-1,1]}$ is odd, $P_{[-1,1]}(-\bp)(1-x,y)=-P_{[-1,1]}(-\bp)(x,y)$. Integrating this relation over $\O$ shows that
$P_{[-1,1]}(-\bp)$ has a zero average and thus, by Lemma \ref{u_proj_c}, that $\oc=0$. We thus see that $\bu= P_{[-1,1]}\left( - \bp\right)$.

\medskip

\textbf{Conforming FEM:} The discrete solution is computed on a family of uniform grids with mesh sizes $h = \frac{1}{2^i}, i=2,\ldots,6$. Due to the symmetry of the mesh and of the solution, approximate solutions are also symmetric and thus have zero average at an order compatible with the stopping criterion in the active set algorithm (the discrete solutions of \eqref{discrete_kkt} are only approximated by this algorithm), see Table \ref{table:algo:FEM_unitsq}. As also seen in this table, the number of iterations of the modified active set algorithm remains very small, and independent on the mesh size. The error estimates and the convergence rates of the control, the post-processed control, the state and the adjoint variables are presented in Table \ref{table:error:FEM_unitsq}. The numerical results corroborate Theorem \ref{theorem.control}, Theorem \ref{thm.fullsuper-convergence} and Corollary \ref{cor.super-convergence}.

\begin{table}[h!!]
	\caption{\small{Example 1, conforming FEM }}
	{\small{\footnotesize
			\begin{center}
				\begin{tabular}{ |c|c|c | c|c| }
					\hline
					$h$ &$U_{a}$ &$f_{a}$ & $Y_{a}$  &$\texttt{ni}$ \\			
					\hline\\[-7pt]  &&\\[-13pt]
					0.250000& 0.002752$\times 10^{-13}$&  0.20699$\times 10^{-14}$& -0.008396$\times 10^{-13}$ &  2\\
					0.125000&-0.008049$\times 10^{-13}$&  0.20912$\times 10^{-14}$& -0.004684$\times 10^{-13}$&   3\\
					0.062500&-0.001370$\times 10^{-13}$&  0.20548$\times 10^{-14}$&  0.010486$\times 10^{-13}$& 3\\
					0.031250&-0.032432$\times 10^{-13}$&  0.21299$\times 10^{-14}$&  0.050725$\times 10^{-13}$&3\\
					0.015625&-0.917129$\times 10^{-13}$&  0.20367$\times 10^{-14}$& -0.495753$\times 10^{-13}$ &  3\\
					\hline 
				\end{tabular}
			\end{center}}}\label{table:algo:FEM_unitsq}
		\end{table}
		\begin{table}[h!!]
			\caption{\small{Convergence results, Example 1, conforming FEM}}
			{\small{\footnotesize
					\begin{center}
						\begin{tabular}{ |c|c | c | c|c|c|c|}
							\hline
							$h$ &$\err_\disc(\by)$ & Order  &$\err_\disc(\nabla\by)$ & Order &$\err_\disc(\bp)$ & Order \\ 			
							\hline\\[-7pt]  &&\\[-13pt]
							0.250000&   0.325104&        -&   0.293424&        -&   0.333213&      -\\
							0.125000&   0.086450&   1.9110&   0.129922&   1.1753&   0.089153&   1.9021\\
							0.062500&   0.022176&   1.9628&   0.064767&   1.0043&   0.022967&   1.9567\\
							0.031250&   0.005591&   1.9879&   0.032578&   0.9914&   0.005798&   1.9860\\
							0.015625&   0.001402&   1.9960&   0.016337&   0.9958&   0.001453&   1.9960\\
							\hline 
						\end{tabular}
					\end{center} 
					\begin{center}
						\begin{tabular}{ |c|c | c |c|c|c|c| }
							\hline
							$h$ &$\err_\disc(\nabla\bp)$ & Order&$\err(\bu)$ & Order  & $\err(\tu)$ & Order  \\
							\hline\\[-7pt]  &&\\[-13pt]	
							0.250000&   0.300144&        -&   0.464300&        -&   0.222006&       -\\
							0.125000&   0.131070&   1.1953&   0.254036&   0.8700&   0.065430&   1.7626\\
							0.062500&   0.064930&   1.0134&   0.126358&   1.0075&   0.016668&   1.9728\\
							0.031250&   0.032599&   0.9941&   0.063453&   0.9938&   0.004226&   1.9797\\
							0.015625&   0.016339&   0.9965&   0.031778&   0.9977&   0.001047&   2.0136\\
							\hline
						\end{tabular}
					\end{center}	}}\label{table:error:FEM_unitsq}
				\end{table}	

\textbf{Non-Conforming FEM:} For comparison, we compute the solutions of the nc$\mathbb{P}_1$ finite element method on the same grids. As for conforming FE, the symmetry of the problem ensures that the approximation solutions have a zero average at an order dictated by the stopping criterion used in the active set algorithm. The results in Tables \ref{table:algo:ncFEM_unitsq} and \ref{table:error:ncFEM_unitsq} are similar to those obtained with the conforming FE.

\begin{table}[h!!]
	\caption{\small{Example 1,	nc$\mathbb{P}_1$FEM }}
	{\small{\footnotesize
			\begin{center}
				\begin{tabular}{ |c|c|c | c|c| }
					\hline                         
					$h$ &$U_{a}$ &$f_{a}$ & $Y_{a}$  &$\texttt{ni}$ \\			
					\hline\\[-6pt]  &&\\[-13pt]
					0.250000& -0.003919$\times 10^{-13}$& 0.206991$\times 10^{-14}$& 0.000518$\times 10^{-12}$&  3\\
					0.125000& -0.067706$\times 10^{-13}$& 0.209121$\times 10^{-14}$& -0.003856$\times 10^{-12}$& 3\\
					0.062500&  0.030900$\times 10^{-13}$& 0.205478$\times 10^{-14}$&  0.017217$\times 10^{-12}$& 3\\
					0.031250& -0.075427$\times 10^{-13}$& 0.212989$\times 10^{-14}$& -0.041154$\times 10^{-12}$& 3\\
					0.015625& -0.187208$\times 10^{-13}$& 0.203674$\times 10^{-14}$&  0.933499$\times 10^{-12}$&  3\\
					\hline 
				\end{tabular}
			\end{center}}}\label{table:algo:ncFEM_unitsq}
		\end{table}
		\begin{table}[h!!]
			\caption{\small{Convergence results, Example 1, nc$\mathbb{P}_1$FEM}}
			{\small{\footnotesize
					\begin{center}
						\begin{tabular}{ |c|c | c | c|c|c|c|}
							\hline
							$h$ &$\err_\disc(\by)$ & Order  &$\err_\disc(\nabla\by)$ & Order &$\err_\disc(\bp)$ & Order \\ 			
							\hline\\[-7pt]  &&\\[-13pt]
							0.250000&   0.148286&        -&   0.409750&        -&   0.146306&    -\\
							0.125000&   0.033274&   2.1559&   0.189599&   1.1118&   0.033499&   2.1268\\
							0.062500&   0.008134&   2.0324&   0.093105&   1.0260&   0.008122&   2.0443\\
							0.031250&   0.002023&   2.0077&   0.046348&   1.0064&   0.002025&   2.0036\\
							0.015625&   0.000505&   2.0019&   0.023148&   1.0016&   0.000505&   2.0041\\
							\hline 
						\end{tabular}
					\end{center} 
					\begin{center}
						\begin{tabular}{ |c|c | c |c|c|c|c| }
							\hline
							$h$ &$\err_\disc(\nabla\bp)$ & Order&$\err(\bu)$ & Order  & $\err(\tu)$ & Order  \\
							\hline\\[-7pt]  &&\\[-13pt]	
							0.250000&   0.408120&        -&   0.473176&        -&   0.284795&     -\\
							0.125000&   0.189770&   1.1047&   0.250457&   0.9178&   0.071206&   1.9999\\
							0.062500&   0.093102&   1.0274&   0.126078&   0.9902&   0.017716&   2.0069\\
							0.031250&   0.046349&   1.0063&   0.063407&   0.9916&   0.004440&   1.9965\\
							0.015625&   0.023149&   1.0016&   0.031770&   0.9970&   0.001109&   2.0007\\
							\hline
						\end{tabular}
					\end{center}	}}\label{table:error:ncFEM_unitsq}
				\end{table}	
\medskip

\textbf{HMM scheme:} This scheme was tested on a series of regular triangular meshes from \cite{benchmark} where the points $\mathcal{P}$ (see \cite[Definition 2.21]{DEH15}) are located at the center of mass of the cells. These meshes are no longer symmetric and thus the symmetry of the approximate solution is lost. Zero averages are thus obtained up to quadrature error, see Table \ref{table:algo:HMM_unitsq}. It has been proved in \cite{bre-05-fam,jd_nn} that the state and adjoint equations enjoy a super-convergence property in $L^2$ norm for such a sequence of meshes; hence, as expected from Theorem \ref{thm.fullsuper-convergence}, so does the scheme for the entire control problem after post-processing of the control. The errors in the energy norm and the $L^2$ norm, together with their orders of convergence, are presented in Table \ref{table:error:HMM_unitsq}.

	\begin{table}[h!!]
					\caption{\small{Example 1, HMM}}
					{\small{\footnotesize
							\begin{center}
								\begin{tabular}{ |c|c |c| c|c| }
									\hline
$h$ &$U_{a}$ & $f_{a}$ &$Y_{a}$  &$\texttt{ni}$ \\			
									\hline\\[-7pt]  &&\\[-13pt]
0.250000& -0.016326& 0.016324& -0.017271& 4\\
0.125000& -0.005300& 0.005300& -0.004968& 4\\
0.062500& -0.001503& 0.001503& -0.001277&  3\\
0.031250& -0.000352& 0.000352& -0.000321&  3\\
									\hline 
								\end{tabular}
							\end{center}}}\label{table:algo:HMM_unitsq}
		\end{table}
\begin{table}[h!!]
			\caption{\small{Convergence results, Example 1, HMM}}
			{\small{\footnotesize
					\begin{center}
						\begin{tabular}{ |c|c | c | c|c|c|c|}
							\hline
							$h$ &$\err_\disc(\by)$ & Order  &$\err_\disc(\nabla\by)$ & Order &$\err_\disc(\bp)$ & Order \\ 			
							\hline\\[-7pt]  &&\\[-13pt]
0.250000& 0.025586&       -&  0.143963&       -&   0.033104&     -\\
0.125000& 0.006764&  1.9194&  0.070970&  1.0204&   0.010044& 1.7207\\
0.062500& 0.001709&  1.9847&  0.035358&  1.0052&   0.002443& 2.0397\\
0.031250& 0.000429&  1.9958&  0.017663&  1.0013&   0.000619& 1.9811\\
							\hline 
						\end{tabular}
					\end{center} 
					\begin{center}
						\begin{tabular}{ |c|c | c |c|c|c|c| }
							\hline
							$h$ &$\err_\disc(\nabla\bp)$ & Order&$\err(\bu)$ & Order  & $\err(\tu)$ & Order  \\
							\hline\\[-7pt]  &&\\[-13pt]	
0.250000&  0.144012&     - &  0.214573&      -  &  0.034890&      -\\
0.125000&  0.070972& 1.0209&  0.109352&   0.9725&  0.009603& 1.8613\\
0.062500&  0.035359& 1.0052&  0.055045&   0.9903&  0.002403& 1.9989\\
0.031250&  0.017663& 1.0013&  0.027551&   0.9985&  0.000605& 1.9893\\
							\hline
						\end{tabular}
					\end{center}	}}\label{table:error:HMM_unitsq}
				\end{table}	
For all three methods (conforming $\mathbb{P}_1$ FE, nc$\mathbb{P}_1$ FE and HMM), the theoretical rates of convergence are confirmed by the numerical outputs. Without post-processing, an $\mathcal O(h)$ convergence rate is obtained on the controls, which validates Theorem \ref{theorem.control}. With post-processing of the controls, the order of convergence of Theorem \ref{thm.fullsuper-convergence} is recovered. We also notice that the super-convergence on the state and adjoint stated in Corollary \ref{cor.super-convergence} is confirmed, provided that the exact state and adjoint are properly projected (usage of the functions $\by_\mesh$ and $\bp_\mesh$ in $\err_\disc(\by)$ and $\err_\disc(\bp)$).

\begin{remark}\label{dependency.rho}
As seen in Table \ref{table:algo:HMM_unitsq}, the modified active set algorithm converges in very few iterations if $\rho=10^{-4}$. We however found that, if $\rho=1$, the modified active set algorithm no longer converges. Further work will investigate in more depth the convergence analysis of the modified active set algorithm, to understand better its dependency with respect to $\rho$.
\end{remark}

		
\subsubsection{\textbf{Example 2 :}} $A=100Id$, $\O$ = $\left(0,1\right)^2$, $\rho =10^{-2}$, $a=-1$, $b=1$.

In this subsection, we present some numerical results for the control problem defined on the unit square domain $\O = \left(0,1\right)^2$ and $A=100Id$. As explained in Example 1, $a=-1$ and $b=1$ imply $\oc=0$.

\medskip

\textbf{Conforming FEM:} We provide in Table \ref{table:algo:FEM_A} the  details of active set algorithm for the conforming finite element method. As expected, the symmetries of the problem provide approximate solutions with a nearly perfect average. For such grids, we obtain super-convergence result for the post-processed control. The errors between the true and computed solutions are computed for different mesh sizes and presented in Table \ref{table:error:FEM_A}. They still follow the expected theoretical rates, and the number of iterations of the
active set algorithm remain small.
	
\begin{table}[h!!]
	\caption{\small{Example 2, conforming FEM }}
	{\small{\footnotesize
			\begin{center}
				\begin{tabular}{ |c|c|c | c|c| }
					\hline
					$h$ &$U_{a}$ &$f_{a}$ & $Y_{a}$  &$\texttt{ni}$ \\			
					\hline\\[-7pt]  &&\\[-13pt]
0.250000&  0.002280$\times 10^{-11}$& 0.209361$\times 10^{-12}$  &   0.009117$\times 10^{-11}$&   2\\
0.125000&  0.018065$\times 10^{-11}$& 0.209375$\times 10^{-12}$  &   0.024496$\times 10^{-11}$&   3\\
0.062500&  0.027564$\times 10^{-11}$& 0.209400$\times 10^{-12}$  &  -0.012778$\times 10^{-11}$&   3\\
0.031250& -0.103755$\times 10^{-11}$& 0.209420$\times 10^{-12}$  &  -0.028850$\times 10^{-11}$&   3\\
0.015625& -0.168277$\times 10^{-11}$& 0.209297$\times 10^{-12}$  &   0.624160$\times 10^{-11}$&   3\\
					\hline 
				\end{tabular}
			\end{center}}}\label{table:algo:FEM_A}
		\end{table}
		\begin{table}[h!!]
			\caption{\small{Convergence results, Example 2,  conforming FEM}}
			{\small{\footnotesize
					\begin{center}
						\begin{tabular}{ |c|c | c | c|c|c|c|}
							\hline
							$h$ &$\err_\disc(\by)$ & Order  &$\err_\disc(\nabla\by)$ & Order &$\err_\disc(\bp)$ & Order \\ 			
							\hline\\[-7pt]  &&\\[-13pt]
							0.250000&   0.328223&        -&   0.296014&        -&   0.328304&   -\\
							0.125000&   0.087182&   1.9126&   0.130232&   1.1846&   0.087209&   1.9125\\
							0.062500&   0.022409&   1.9600&   0.064814&   1.0067&   0.022417&   1.9599\\
							0.031250&   0.005653&   1.9870&   0.032584&   0.9921&   0.005655&   1.9870\\
							0.015625&   0.001417&   1.9963&   0.016338&   0.9960&   0.001417&   1.9963\\
							\hline 
						\end{tabular}
					\end{center} 
					\begin{center}
						\begin{tabular}{ |c|c | c |c|c|c|c| }
							\hline
							$h$ &$\err_\disc(\nabla\bp)$ & Order&$\err(\bu)$ & Order  & $\err(\tu)$ & Order  \\
							\hline\\[-7pt]  &&\\[-13pt]	
							0.250000&   0.296080&        -&   0.463836&        -&   0.218080&          -\\
							0.125000&   0.130243&   1.1848&   0.253857&   0.8696&   0.064390&   1.7600\\
							0.062500&   0.064816&   1.0068&   0.126333&   1.0068&   0.016358&   1.9768\\
							0.031250&   0.032584&   0.9922&   0.063449&   0.9936&   0.004145&   1.9807\\
							0.015625&   0.016338&   0.9960&   0.031778&   0.9976&   0.001026&   2.0139\\
							\hline
						\end{tabular}
					\end{center}	}}\label{table:error:FEM_A}
				\end{table}

\medskip

\textbf{Non-Conforming FEM:} The results, presented in Tables \ref{table:algo:ncFEM_A} and \ref{table:error:ncFEM_A}, are similar to those for the conforming FE scheme.

\begin{table}[h!!]
	\caption{\small{Example 2, nc$\mathbb{P}_1$FEM }}
	{\small{\footnotesize
			\begin{center}
				\begin{tabular}{ |c|c|c | c|c| }
					\hline
					$h$ &$U_{a}$ &$f_{a}$ & $Y_{a}$  &$\texttt{ni}$ \\			
					\hline\\[-7pt]  &&\\[-13pt]
  0.250000&  -0.000977$\times 10^{-10}$ &  0.209361$\times 10^{-12}$  & 0.000739 $\times 10^{-10}$  &   3\\
  0.125000&  -0.006518$\times 10^{-10}$&   0.209375$\times 10^{-12}$  &-0.000960$\times 10^{-10}$   &   3\\
  0.062500&  -0.004320$\times 10^{-10}$&   0.209400$\times 10^{-12}$ &  0.002330$\times 10^{-10}$  &   3\\
  0.031250&  -0.007236$\times 10^{-10}$ &  0.209420 $\times 10^{-12}$ & 0.054029 $\times 10^{-10}$  &   3\\
  0.015625&   0.346321$\times 10^{-10}$ &   0.209297$\times 10^{-12}$ &  0.276914 $\times 10^{-10}$  &   3\\
					\hline 
				\end{tabular}
			\end{center}}}\label{table:algo:ncFEM_A}
		\end{table}
		\begin{table}[h!!]
			\caption{\small{Convergence results, Example 2, nc$\mathbb{P}_1$FEM}}
			{\small{\footnotesize
					\begin{center}
						\begin{tabular}{ |c|c | c | c|c|c|c|}
							\hline
							$h$ &$\err_\disc(\by)$ & Order  &$\err_\disc(\nabla\by)$ & Order &$\err_\disc(\bp)$ & Order \\ 			
							\hline\\[-7pt]  &&\\[-13pt]
							0.250000&   0.148286&        -&   0.409750&        -&   0.148262&        -\\
							0.125000&   0.033263&   2.1564&   0.189599&   1.1118&   0.033265&   2.1561\\
							0.062500&   0.008131&   2.0324&   0.093105&   1.0260&   0.008131&   2.0325\\
							0.031250&   0.002022&   2.0077&   0.046348&   1.0064&   0.002022&   2.0076\\
							0.015625&   0.000505&   2.0019&   0.023148&   1.0016&   0.000505&   2.0019\\
							\hline 
						\end{tabular}
					\end{center} 
					\begin{center}
						\begin{tabular}{ |c|c | c |c|c|c|c| }
							\hline
							$h$ &$\err_\disc(\nabla\bp)$ & Order&$\err(\bu)$ & Order  & $\err(\tu)$ & Order  \\
							\hline\\[-7pt]  &&\\[-13pt]	
							0.250000&   0.409732&        -&   0.473136&        -&   0.286537&       -\\
							0.125000&   0.189600&   1.1117&   0.250390&   0.9181&   0.071004&   2.0128\\
							0.062500&   0.093105&   1.0260&   0.126079&   0.9899&   0.017719&   2.0026\\
							0.031250&   0.046348&   1.0064&   0.063407&   0.9916&   0.004439&   1.9970\\
							0.015625&   0.023148&   1.0016&   0.031770&   0.9970&   0.001109&   2.0004\\	
							\hline
						\end{tabular}
					\end{center}	}}\label{table:error:ncFEM_A}
				\end{table}

\textbf{HMM scheme: }
The results are presented in Tables \ref{table:algo:HMM_A} and \ref{table:error:HMM_A}. They are qualitatively similar to those for Example 1. As mentioned before, the algorithm is not convergent for $\rho=1$.

		\begin{table}[h!!]
			\caption{\small{Example 2, HMM}} 
			{\small{\footnotesize
					\begin{center}
						\begin{tabular}{ |c|c|c | c|c| }
							\hline
							$h$ &$U_{a}$ &$f_{a}$ & $Y_{a}$  &$\texttt{ni}$ \\			
							\hline\\[-7pt]  &&\\[-13pt]
					0.250000&   -1.000000 &  1.817180& 81.718002 &   -\\
					0.125000&   -0.528617 &  0.523335& -0.528289  &  6\\
					0.062500&   -0.136036 &  0.134678& -0.135812 &   5\\
					0.031250&   -0.034208 &  0.033866& -0.034178 &   5\\	
							\hline 
						\end{tabular}
					\end{center}}}\label{table:algo:HMM_A}
				\end{table}
				\begin{table}[h!!]
					\caption{\small{Convergence results, Example 2, HMM}}
					{\small{\footnotesize
							\begin{center}
								\begin{tabular}{ |c|c | c | c|c|c|c|}
									\hline
									$h$ &$\err_\disc(\by)$ & Order  &$\err_\disc(\nabla\by)$ & Order &$\err_\disc(\bp)$ & Order \\ 			
									\hline\\[-7pt]  &&\\[-13pt]
	0.250000 &  81.717083 &         -&       0.143996  &     -  &  392673.3&       - \\
 	0.125000 &   0.528572 &   7.2724 &       0.070971 &   1.0207&    0.876152&    18.7737\\
    0.062500 &   0.135884 &   1.9597 &       0.035359 &   1.0052&    0.216641&  2.0159\\
	0.031250 &   0.034196 &   1.9905 &       0.017663 &   1.0013&    0.054238 & 1.9979\\			
									\hline 
								\end{tabular}
							\end{center} 
							\begin{center}
								\begin{tabular}{ |c|c | c |c|c|c|c| }
									\hline
									$h$ &$\err_\disc(\nabla\bp)$ & Order&$\err(\bu)$ & Order  & $\err(\tu)$ & Order  \\
									\hline\\[-7pt]  &&\\[-13pt]	
		0.250000 &   0.143987 &        -&    1.686504 &     -   &    1.691346 &       - \\
		0.125000 &   0.070970 &   1.0207&    0.878849 &   0.9403&    0.874778 &     0.9512 \\
		0.062500 &   0.035358 &   1.0052&    0.237129 &   1.8899&    0.230873 &    1.9218 \\
		0.031250 &   0.017663 &   1.0013&    0.064673 &   1.8744&    0.058612 &    1.9778\\
							\hline
								
								\end{tabular}
							\end{center}	}}\label{table:error:HMM_A}
						\end{table}			
														
		
\subsubsection{\textbf{Example 3 :}} $\O=(0,1)^2$, $ \rho=10^{-4},$ $a=-0.5,\, b=1$.
In this case, since $P_{[a,b]}$ is no longer odd, $P_{[a,b]}(-\bp)$ no longer has a zero average and, to compute $\err_\disc(\bu)$, we need to find $\oc$ such that $\dashint_\O P_{[a,b]}(-\bp+\oc)\d\x=0$. This $\oc$ can be found by a bisection method, by computing the averages on a very thin mesh and bisecting until we find a proper $\oc$. Using a mesh of size $h= 0.00195$, we find  $c \approx -0.24596797$.

\medskip

\textbf{Conforming FEM:} The numerical results obtained using conforming finite element method are shown in Tables \ref{table:algo:FEM_a} and \ref{table:error:FEM_a} respectively. Since there is a loss of symmetry, the approximate solutions have zero averages only up to quadrature error (compare $U_a$ and $f_a$ in Table \ref{table:algo:FEM_a}). 
Here, we observed that the modified active set algorithm converges only when $\rho\le 10^{-1}$. When it does,
though, the number of iterations remain very small. As in Examples 1 and 2, the theoretical rates of convergence are confirmed by these numerical outputs.

\begin{table}[h!!]
	\caption{\small{Example 3, conforming FEM }}
	{\small{\footnotesize
			\begin{center}
				\begin{tabular}{ |c|c|c | c|c| }
					\hline
					$h$ &$U_{a}$ &$f_{a}$ & $Y_{a}$  &$\texttt{ni}$ \\			
					\hline\\[-7pt]  &&\\[-13pt]
					0.250000&   0.0020160&  -0.0020160&    0.201602$\times 10^{-6}$&   4\\
					0.125000&   0.0055595&  -0.0055595&    0.555952$\times 10^{-6}$&   4\\
					0.062500&  -0.0004794&   0.0004795&   -0.047944$\times 10^{-6}$&   4\\
					0.031250&   0.0001470&  -0.0001470&    0.014705$\times 10^{-6}$&   5\\
					0.015625&  -0.0000136&   0.0000136&   -0.001362$\times 10^{-6}$&   5\\
					\hline 
				\end{tabular}
			\end{center}}}\label{table:algo:FEM_a}
		\end{table}
		\begin{table}[h!!]
			\caption{\small{Convergence results, Example 3, conforming FEM}}
			{\small{\footnotesize
					\begin{center}
						\begin{tabular}{ |c|c | c | c|c|c|c|}
							\hline
							$h$ &$\err_\disc(\by)$ & Order  &$\err_\disc(\nabla\by)$ & Order &$\err_\disc(\bp)$ & Order \\ 			
							\hline\\[-7pt]  &&\\[-13pt]
							0.250000&   0.325266&        -&   0.293567&       -  &   0.346894&         - \\
							0.125000&   0.086733&   1.9070&   0.130041&    1.1747&   0.097046&   1.8378\\
							0.062500&   0.022291&   1.9601&   0.064790&    1.0051&   0.025081&   1.9521\\
							0.031250&   0.005624&   1.9868&   0.032581&    0.9917&   0.006219&   2.0117\\
							0.015625&   0.001410&   1.9963&   0.016337&    0.9959&   0.001569&   1.9865\\
							\hline 
						\end{tabular}
					\end{center} 
					\begin{center}
						\begin{tabular}{ |c|c | c |c|c|c|c| }
							\hline
							$h$ &$\err_\disc(\nabla\bp)$ & Order&$\err(\bu)$ & Order  & $\err(\tu)$ & Order  \\
							\hline\\[-7pt]  &&\\[-13pt]	
							0.250000&   0.300149&        -&    0.466701&      -  &   0.234197&       -\\
							0.125000&   0.131075&   1.1953&    0.268982&   0.7950&   0.064265&   1.8656\\
							0.062500&   0.064931&   1.0134&    0.138258&   0.9602&   0.016053&   2.0012\\
							0.031250&   0.032599&   0.9941&    0.069620&   0.9898&   0.003996&   2.0064\\
							0.015625&   0.016339&   0.9965&    0.034944&   0.9945&   0.001002&   1.9950\\
							\hline
						\end{tabular}
					\end{center}	}}\label{table:error:FEM_a}
				\end{table}	

\medskip

\textbf{Non-Conforming FEM:} The results are similar to those obtained with the conforming FE method (see Tables \ref{table:algo:ncFEM_a} and \ref{table:error:ncFEM_a}).

\begin{table}[h!!]
	\caption{\small{Example 3, nc$\mathbb{P}_1$FEM}}
	{\small{\footnotesize
			\begin{center}
				\begin{tabular}{ |c|c|c | c|c| }
					\hline
					$h$ &$U_{a}$ &$f_{a}$ & $Y_{a}$  &$\texttt{ni}$ \\			
					\hline\\[-7pt]  &&\\[-13pt]
					0.250000&   0.002016&  -0.002016&   0.0201601$\times 10^{-5}$&  4\\
					0.125000&   0.005560&  -0.005559&  -0.1301803$\times 10^{-5}$&   5\\
					0.062500&  -0.000480&   0.000479&  -0.0017424$\times 10^{-5}$&   5\\
					0.031250&   0.000147&  -0.000147&   0.0011093$\times 10^{-5}$&   5\\
					0.015625&  -0.000014&   0.000014&  -0.0001436$\times 10^{-5}$&   5\\
					\hline 
				\end{tabular}
			\end{center}}}\label{table:algo:ncFEM_a}
		\end{table}
		\begin{table}[h!!]
			\caption{\small{Convergence results, Example 3, nc$\mathbb{P}_1$FEM}}
			{\small{\footnotesize
					\begin{center}
						\begin{tabular}{ |c|c | c | c|c|c|c|}
							\hline
							$h$ &$\err_\disc(\by)$ & Order  &$\err_\disc(\nabla\by)$ & Order &$\err_\disc(\bp)$ & Order \\ 			
							\hline\\[-7pt]  &&\\[-13pt]
							0.250000&   0.148286&        -&    0.409750&        -&   0.141781&            -\\
							0.125000&   0.033270&   2.1561&    0.189600&   1.1118&   0.032889&   2.1080\\
							0.062500&   0.008133&   2.0324&    0.093105&   1.0260&   0.008041&   2.0322\\
							0.031250&   0.002022&   2.0077&    0.046348&   1.0064&   0.001994&   2.0118\\
							0.015625&   0.000505&   2.0019&    0.023148&   1.0016&   0.000498&   2.0015\\
							\hline 
						\end{tabular}
					\end{center} 
					\begin{center}
						\begin{tabular}{ |c|c | c |c|c|c|c| }
							\hline
							$h$ &$\err_\disc(\nabla\bp)$ & Order&$\err(\bu)$ & Order  & $\err(\tu)$ & Order  \\
							\hline\\[-7pt]  &&\\[-13pt]	
							0.250000&   0.408120&        -&   0.494425&        -&   0.269888&    -\\
							0.125000&   0.189770&   1.1047&   0.269866&   0.8735&   0.080165&   1.7513\\
							0.062500&   0.093102&   1.0274&   0.138223&   0.9652&   0.019967&   2.0054\\
							0.031250&   0.046349&   1.0063&   0.069625&   0.9893&   0.005091&   1.9715\\
							0.015625&   0.023149&   1.0016&   0.034941&   0.9947&   0.001283&   1.9883\\
							\hline
						\end{tabular}
					\end{center}	}}\label{table:error:ncFEM_a}
				\end{table}	

\medskip

\textbf{HMM scheme:} Tables \ref{table:algo:HMM_a} and \ref{table:error:HMM_a} show that the HMM scheme behave similarly to the FE schemes. Note that, here too, the convergence of the modified active set algorithm is only observed if $\rho\le 10^{-1}$.

	\begin{table}[h!!]
		\caption{\small{Example 3, HMM }}
		{\small{\footnotesize
				\begin{center}
					\begin{tabular}{ |c|c|c | c|c| }
						\hline
						$h$ &$U_{a}$ &$f_{a}$ & $Y_{a}$  &$\texttt{ni}$ \\			
						\hline\\[-7pt]  &&\\[-13pt]
  0.250000&  -0.019043&   0.019041&  -0.017271&   5\\
  0.125000&  -0.005459&   0.005459&  -0.004968&   5\\
  0.062500&  -0.001300&   0.001300&  -0.001277&   5\\
  0.031250&  -0.000331&   0.000331&  -0.000321&   5\\
						
						\hline 
					\end{tabular}
				\end{center}}}\label{table:algo:HMM_a}
			\end{table}
			\begin{table}[h!!]
				\caption{\small{Convergence results, Example 3, HMM }}
				{\small{\footnotesize
						\begin{center}
							\begin{tabular}{ |c|c | c | c|c|c|c|}
								\hline
								$h$ &$\err_\disc(\by)$ & Order  &$\err_\disc(\nabla\by)$ & Order &$\err_\disc(\bp)$ & Order \\ 			
								\hline\\[-7pt]  &&\\[-13pt]
 0.250000&   0.026037&        -&   0.144014&        -&   0.055044&    - \\
 0.125000&   0.006841&   1.9284&   0.070972&   1.0209&   0.013361&   2.0425\\
 0.062500&   0.001728&   1.9853&   0.035359&   1.0052&   0.003342&   1.9995\\
 0.031250&   0.000433&   1.9956&   0.017663&   1.0013&   0.000843&   1.9869\\
								\hline 
							\end{tabular}
						\end{center} 
						\begin{center}
							\begin{tabular}{ |c|c | c |c|c|c|c| }
								\hline
								$h$ &$\err_\disc(\nabla\bp)$ & Order&$\err(\bu)$ & Order  & $\err(\tu)$ & Order  \\
								\hline\\[-7pt]  &&\\[-13pt]	
 0.250000&   0.144013&        -&   0.237647&       - &   0.050184&        -\\
 0.125000&   0.070972&   1.0209&   0.120112&   0.9844&   0.013482&   1.8962\\
 0.062500&   0.035359&   1.0052&   0.061226&   0.9722&   0.003468&   1.9586\\
 0.031250&   0.017663&   1.0013&   0.030583&   1.0014&   0.000872&   1.9916\\
								\hline
							\end{tabular}
						\end{center}	}}\label{table:error:HMM_a}
					\end{table}	


\section{Appendix} \label{sec:appendix}

The proofs of error estimates for control, state and adjoint variables are obtained by modifying the proofs of the corresponding results in \cite{GDM_control}. For the sake of completeness and readability, we provide here detailed proofs, highlighting in chosen places where modifications are required due to the pure Neumann boundary
conditions (which mostly amount to making sure that certain averages have been properly fixed).

\begin{proof}[Proof of Theorem \ref{theorem.control}]
	Define the following auxiliary discrete
	problem: 
	\begin{subequations} \label{discrete_aux}
		\begin{align}
&\mbox{Seek $(y_{\disc}(\bu), p_{\disc}(\bu) )\in X_{\disc,\star} \times X_{\disc} $ such that}\nonumber\\
& a_{\disc}(y_{\disc}(\bu),w_{\disc}) = (f + \bu, \Pi_\disc w_{\disc})\qquad
		\forall \: w_{\disc} \in  X_{\disc,\star}, \label{state_aux}\\
		& a_{\disc}(z_{\disc}, p_{\disc}(\bu)) = (\by-\yd, \Pi_\disc z_{\disc})\qquad  \forall \: z_{\disc} \in  X_{\disc}, \label{adj_aux}
		\end{align}
	\end{subequations} 
	where the co-state $p_{\disc}(\bu)$ is chosen such that $\dashint_{\O}^{}\Pi_\disc p_\disc(\bu) \d\x = \dashint_{\O}^{}\Pi_\disc \bp_\disc \d\x=\dashint_{\O}^{} \bp\d\x$. For Neumann boundary conditions, this particular choice is essential as it ensures that $p_\disc(\bu)-\bp_\disc\in X_{\disc,\star}$
can be used as a test function $w_\disc$ in \eqref{discrete_state} and \eqref{state_aux}.
	Recalling that $\Proj$ is the orthogonal projection on piecewise constant functions on $\mesh$, we obtain $\Proj(\Uad) \subset \Uh$. Also, for $u \in \Uad$ and $K \in \mesh$, $\Proj u_{|K} =\dashint_K u\d\x \in [a,b]$ and,  using \eqref{comp.cond},
	$$\dashint_\O \Proj u\d\x=\sum_{K \in \mesh}^{}\dashint_K \Proj u\d\x=\sum_{K \in \mesh}^{}\dashint_K u \d\x=\dashint_\O u \d\x=0.$$
	Hence, $\Proj(\Uad)\subset \Uadh$.
	
	Set $P_{\disc,\alpha}(\bu)=\alpha^{-1}\Pi_\disc p_{\disc}(\bu)$,\,
	$\bP_{\disc,\alpha}=\alpha^{-1}\Pi_\disc \bp_{\disc}$ and  $\bP_{\alpha}=\alpha^{-1} \bp$. 
Since $\bu_h \in \Uadh \subset \Uad$ and $\Proj\bu \in \Uadh$, from the optimality conditions (\eqref{opt_cont} and \eqref{opt_discrete}),
		\begin{align*}
		-\alpha(\bP_{\alpha}+\bu,\bu-\bu_h)\ge{}& 0,\\
		\alpha(\bP_{\disc,\alpha}+\bu_h,\bu - \bu_h) 
		\ge{}& \alpha(\bP_{\disc,\alpha}+\bu_h,\bu- \Proj\bu).
		\end{align*}
Adding these two inequalities yields
		\begin{align}
		\alpha\norm{\bu-\bu_h}{}^2  \le{}&-\alpha(\bP_{\disc,\alpha}+\bu_h,\bu- \Proj(\bu))
		+\alpha(\bP_{\disc,\alpha} - \bP_{\alpha},\bu-\bu_h)
		\nonumber\\
		={}& -\alpha(\bP_{\disc,\alpha} +\bu_h,\bu-\Proj\bu)+\alpha(\bP_{\disc,\alpha}-P_{\disc,\alpha}(\bu),\bu-\bu_h)
	\nonumber\\
		& \qquad-\alpha(\bP_{\alpha}-P_{\disc,\alpha}(\bu),\bu-\bu_h).
		\label{inter3}
		\end{align}
		By orthogonality property of $\Proj$ we have $(\bu_h,\bu-\Proj\bu)=0$ and 
		$(\Proj\bP_{\alpha},\bu-\Proj\bu)=0$. Therefore, the first term in the right-hand side of \eqref{inter3} can be re-cast as
		\begin{align}
		-\alpha(\bP_{\disc,\alpha}+{}&\bu_h,\bu-\Proj\bu)\nonumber\\
		={}&-\alpha(\bP_{\alpha}-\Proj\bP_{\alpha},\bu-\Proj\bu)+\alpha(\bP_{\alpha}-P_{\disc,\alpha}(\bu),\bu-\Proj\bu)\nonumber\\
& +\alpha(P_{\disc,\alpha}(\bu)-\bP_{\disc,\alpha},\bu-\Proj\bu).
		\label{inter4}
		\end{align}
		By Cauchy-Schwarz inequality, the first term on the right hand side of {\eqref{inter4}} is estimated as 
		\be\label{est.T1}
		-\alpha(\bP_{\alpha}-\Proj\bP_{\alpha},\bu-\Proj\bu)\le E_h(\bp)E_h(\bu).
		\ee
		Equation \eqref{adj_aux} shows that $p_\disc(\bu)$ is the
		solution of the GS corresponding to the adjoint
		problem \eqref{adj_cont}, whose solution is $\bp$. Therefore, using the fact that  $\dashint_{\O}^{}\Pi_\disc p_\disc(\bu) \d\x = \dashint_{\O}^{} \bp\d\x$ (note that the specific
		relation between the continuous and discrete co-states is essential here), by Theorem \ref{th:error.est.PDE},
		\begin{equation}
		\norm{\bP_{\alpha}-P_{\disc,\alpha}(\bu)}{}
		=\alpha^{-1}\norm{\bp-\Pi_\disc p_\disc(\bu)}{}\lesssim \alpha^{-1}\WS_\disc(\bp).
		\label{est.bPPdisc}
		\end{equation}
		Hence, using the Cauchy--Schwarz inequality,
		\be
	\alpha(\bP_{\alpha}-P_{\disc,\alpha}(\bu),\bu-\Proj\bu)\lesssim \WS_\disc(\bp)E_h(\bu).
		\label{est.T2}
		\ee
		 Using the definitions of $C_\disc$, $\norm{\cdot}{\disc}$ and the fact that $p_\disc(\bu)-\bp_\disc \in X_{\disc,\star}$, we find that
				\begin{align}
				\norm{\Pi_\disc p_\disc(\bu)-\Pi_\disc \bp_\disc}{}^2 \lesssim{}&\norm{ p_\disc(\bu)-\bp_\disc}{\disc}^2=\norm{\nabla_\disc p_\disc(\bu)-\nabla_\disc \bp_\disc}{}^2.\label{est.discreteadj}
				\end{align} By writing the difference of \eqref{adj_aux}
				and \eqref{discrete_adjoint} we see that $p_\disc(\bu)-\bp_\disc$
				is the solution to the GS \eqref{base.GS} with
				source term $F=\by-\Pi_\disc\by_\disc$. i.e, for all $z_\disc \in X_{\disc}$
				\begin{align*}
				a_\disc(z_\disc, p_\disc(\bu)-\bp_\disc) =& (\by-\Pi_\disc\by_\disc, \Pi_\disc z_\disc).
				\end{align*}
				Choose $z_\disc=p_\disc(\bu)-\bp_\disc$ in the above equality and use it in \eqref{est.discreteadj} to obtain
				\begin{align*}
				\norm{\Pi_\disc p_\disc(\bu)-\Pi_\disc \bp_\disc}{}^2 \lesssim{}\norm{\nabla_\disc p_\disc(\bu)-\nabla_\disc \bp_\disc}{}^2 \lesssim{}\norm{\by-\Pi_\disc\by_\disc}{}\norm{\Pi_\disc p_\disc(\bu)-\Pi_\disc \bp_\disc}{}.
				\end{align*}
			As a consequence,
			\begin{align*}
			\norm{P_{\disc,\alpha}(\bu)-\bP_{\disc,\alpha}}{}
			={}&
		\alpha^{-1}	\norm{\Pi_\disc p_\disc(\bu)-\Pi_\disc \bp_\disc}{}
			\lesssim{} 	\alpha^{-1}\norm{\by-\Pi_\disc\by_\disc}{}\\
			\lesssim{}& 	\alpha^{-1}\norm{\by-\Pi_\disc y_\disc(\bu)}{}
			+ 	\alpha^{-1}\norm{\Pi_\disc y_\disc(\bu)-\Pi_\disc \by_\disc}{}.
			\end{align*}
Use Theorem \ref{th:error.est.PDE} with $\psi=\by$ to bound the first term in the above expression. This along with an application of Young's inequality yields an estimate for the last term in \eqref{inter4} as
			\begin{multline}\label{est.T3}
			\alpha(P_{\disc,\alpha}(\bu)-\bP_{\disc,\alpha},\bu-\Proj\bu)\\
		\le  C_1 E_h(\bu)\WS_\disc(\by)+C_1E_h(\bu)^2
			+\frac{1}{4}\norm{\Pi_\disc y_\disc(\bu)-\Pi_\disc \by_\disc}{}^2
			\end{multline}
			where $C_1$ depends only on $\O$, $A$ and an upper bound of $C_\disc$. Plugging \eqref{est.T1}, \eqref{est.T2} and \eqref{est.T3} in \eqref{inter4} yields
\begin{align}
-\alpha(\bP_{\disc,\alpha}+\bu_h,\bu-\Proj\bu)\le{}& E_h(\bp)E_h(\bu)+C_2E_h(\bu)\WS_\disc(\bp)+ C_1 E_h(\bu)\WS_\disc(\by)\nonumber \\
& +C_1E_h(\bu)^2
		+\frac{1}{4}\norm{\Pi_\disc y_\disc(\bu)-\Pi_\disc \by_\disc}{}^2,\label{est:firstterm}
		\end{align}
where $C_2$ is the hidden constant in \eqref{est.T2}. Let us turn to the second term in the right-hand side of \eqref{inter3}.
		From \eqref{discrete_adjoint} and \eqref{adj_aux}, for all $z_\disc \in X_{\disc}$,
		\be\label{eq:1}
		a_\disc(z_\disc,\bp_\disc- p_\disc(\bu)) = (\Pi_\disc \by_\disc- \by, \Pi_\disc z_\disc). 
		\ee
		Also, from \eqref{discrete_state} and \eqref{state_aux}, for all $w_\disc \in X_{\disc,\star}$,
		\begin{align}
		a_\disc(\by_\disc- y_\disc(\bu), w_\disc) =& (\bu_h-\bu, \Pi_\disc w_\disc).
		\label{eq:2}
		\end{align}
	Choose $z_\disc = \by_\disc-y_\disc(\bu) \in X_{\disc}$ in \eqref{eq:1}, $w_\disc=\bp_\disc- p_\disc(\bu) \in X_{\disc,\star}$ in \eqref{eq:2}, use the symmetry of $a_\disc(\cdot,\cdot)$, Theorem \ref{th:error.est.PDE} with $\psi=\by$ and Young's inequality to obtain
		\begin{align}
		\alpha(\bP_{\disc,\alpha}&-P_{\disc,\alpha}(\bu),\bu-\bu_h)
		={}- (\Pi_\disc \by_\disc- \by, \Pi_\disc \by_\disc- \Pi_\disc y_\disc(\bu))\nonumber\\
		={}& (\by - \Pi_\disc y_\disc(\bu), \Pi_\disc \by_\disc- \Pi_\disc y_\disc(\bu)) - \norm{\Pi_\disc \by_\disc-  \Pi_\disc y_\disc(\bu)}{}^2 \nonumber \\
	\lesssim {}&\WS_\disc(\by)\norm{\Pi_\disc\by_\disc-\Pi_\disc y_\disc(\bu)}{}- \norm{\Pi_\disc \by_\disc-  \Pi_\disc y_\disc(\bu)}{}^2 \nonumber \\
		\le{}&C_3\WS_\disc(\by)^2+
		\frac{1}{4}\norm{\Pi_\disc\by_\disc-\Pi_\disc y_\disc(\bu)}{}^2  - \norm{\Pi_\disc \by_\disc-  \Pi_\disc y_\disc(\bu)}{}^2,
		\label{inter5}
		\end{align}
	 where $C_3$ only depends on $\O$, $A$ and an upper bound of $C_\disc$.	
	The last term in the right hand side of \eqref{inter3} can be estimated using \eqref{est.bPPdisc} and Young's inequality:
	 \be\label{est.T4}
-\alpha(\bP_{\alpha}-P_{\disc,\alpha}(\bu),\bu-\bu_h)\le \frac{\alpha}{2}\norm{\bu- \bu_h}{}^2
	 +C_4\WS_\disc(\bp)^2,
	 \ee
	 	where $C_4$ only depends on $\O$, $A$, $\alpha$ and an upper bound of $C_\disc$.
	Substitute \eqref{est:firstterm}, \eqref{inter5} and \eqref{est.T4} into \eqref{inter3}, apply the Young's inequality and $\sqrt{\sum_i a_i^2}\le \sum_i a_i$ to complete the proof.
\end{proof}

\begin{proof}[Proof of Proposition \ref{prop.state.adj}]
		An application of triangle inequality yields 
	\begin{align}
	\norm{\Pi_\disc \by_\disc-\by}{}+\norm{\nabla_\disc \by_\disc-\nabla\by}{}  
	\le{}&\norm{\Pi_\disc \by_\disc-\Pi_\disc y_\disc(\bu) }{}+ \norm{\nabla_\disc \by_\disc-\nabla_\disc y_\disc(\bu)}{}\nonumber  \\
	&+ \norm{\Pi_\disc y_\disc(\bu)-\by}{}+\norm{\nabla_\disc y_\disc(\bu)-\nabla\by}{} .\label{est.state.err}
	\end{align}
	Subtracting \eqref{discrete_state} and \eqref{state_aux}, and using the stability property of GS (Proposition \ref{prop.stab}), the first two terms in the right hand side of the above inequality can be estimated as
	$$ \norm{\Pi_\disc \by_\disc-\Pi_\disc y_\disc(\bu) }{}+\norm{\nabla_\disc \by_\disc-\nabla_\disc y_\disc(\bu)}{} \lesssim \norm{\bu-\bu_h}{}.$$
	The last two terms on the right hand side of the \eqref{est.state.err} are estimated using Theorem \ref{th:error.est.PDE} as
		$$ \norm{\Pi_\disc y_\disc(\bu)-\by}{}+\norm{\nabla_\disc y_\disc(\bu)-\nabla\by}{} \lesssim \WS_\disc(\by) .$$ 
		A combination of the above two results yields the error estimates \eqref{est.basic.y} for the state variable. 
		A use of $\dashint_{\O}^{}\Pi_\disc p_\disc(\bu) \d\x=\dashint_{\O}^{}\Pi_\disc \bp_\disc \d\x$ in Proposition \ref{prop.stab} leads to the error estimates for the adjoint variable in a similar way.
\end{proof}

\begin{proof}[Proof of Theorem \ref{thm.super-convergence}]
Consider the auxiliary problem defined by: For $g\in L^2(\O)$, let $p_\disc^*(g)\in X_{\disc}$ solve
	\begin{equation} 
	a_\disc(z_\disc,p_\disc^*(g))=(\Pi_\disc y_\disc(g)-\yd,\Pi_\disc z_\disc)\qquad\forall z_\disc\in X_{\disc}, \label{aux2}
	\end{equation}
	where $y_\disc(g)$ is given by \eqref{state_aux} with $\bu$ replaced by $g$. We fix $p_\disc^*(g)$ by imposing $\dashint_{\O}^{}\Pi_\disc p_\disc^*(g) \d\x=\dashint_{\O}^{} \bp\d\x$. This choice is dictated by the pure Neumann boundary condition and will be essential.
	
	For $K\in\mesh$, let $\overline{\x}_K$ be the centroid (centre of mass) of $K$. A standard approximation property (see e.g. \cite[Lemma A.7]{jd_nn} with
	$w_K\equiv 1$) yields
	\begin{equation}\label{approx.xK}
	\forall K\in\mesh\,,\;\forall \phi\in H^2(K)\,,\;
	\norm{\Proj \phi-\phi(\overline{\x}_K)}{L^2(K)}\lesssim_{\eta} {\rm diam}(K)^2 \norm{\phi}{H^2(K)}.
	\end{equation}
	
Define $\hat{u}$ and $\hat{p}$ a.e.\ on $\O$ by:
	For all $K\in\mesh$ and all $\x\in K$,
	$\hat{u}(\x) = \bu(\overline{\x}_K)$ and  $ \hat{p}(\x)=\bp(\overline{\x}_K)$.
	From \eqref{Projection_ppu} and the Lipschitz continuity of $P_{[a,b]}$, we obtain
	\begin{align}
	\norm{\tu-\tu_{h}}{}
	\le{}&\alpha^{-1}\norm{ \Pi_\disc \bp_\disc-\bp_\mesh}{}\nonumber\\
	\le{}&  \alpha^{-1}\norm{ \bp_\mesh-\Pi_\disc p_\disc^*(\bu)}{} +  \alpha^{-1}\norm{\Pi_\disc p_\disc^*(\bu) -\Pi_\disc p_\disc^*(\hat{u})}{}  \nonumber \\
	& + \alpha^{-1}\norm{\Pi_\disc p_\disc^*(\hat{u})-\Pi_\disc \bp_\disc}{}  \nonumber \\
	=:{}&   \alpha^{-1}T_1 + \alpha^{-1}T_2+ \alpha^{-1}T_3.\label{ineq}
	\end{align}
	
	\textbf{Step 1}: estimate of $T_1$.
	
A use of triangle inequality, \eqref{adj_cont}, \eqref{adj_aux} and $\assum{1}$-i) leads to
	\begin{align}
	T_1 & \le  \norm{ \bp_\mesh-\Pi_\disc p_\disc(\bu)}{} + \norm{ \Pi_\disc p_\disc(\bu)-\Pi_\disc p_\disc^*(\bu)}{} \nonumber\\
	&\lesssim   h^2\norm{\by -\yd}{H^1(\O)} +\norm{ \Pi_\disc p_\disc(\bu)-\Pi_\disc p_\disc^*(\bu)}{}.
	\label{a1.inequality}
	\end{align} 
	
	We now estimate the last term in this inequality. Use the definitions of $C_\disc$, $\norm{\cdot}{\disc}$ and the fact that $\dashint_{\O}^{}\Pi_\disc p_\disc (\bu) \d\x=\dashint_{\O}^{}\Pi_\disc p_\disc^*(\bu) \d\x$ to obtain
\be \label{T1.adj}
\norm{\Pi_\disc(p_\disc(\bu)-p_\disc^*(\bu))}{}^2  \lesssim{}	\norm{\nabla_\disc(p_\disc(\bu)-p_\disc^*(\bu))}{}^2.\ee
	Subtract \eqref{aux2} with $g=\bu$ from \eqref{adj_aux}, substitute $z_\disc=p_\disc(\bu)-p_\disc^*(\bu)$, use property \eqref{prop.M.1} in $\assum{1}$-ii) and  Cauchy-Schwarz inequality to obtain
		\begin{align*}
		\norm{\nabla_\disc(p_\disc(\bu)-p_\disc^*(\bu))}{}^2
		& \lesssim{}
		a_\disc(p_\disc(\bu)-p_\disc^*(\bu),p_\disc(\bu)-p_\disc^*(\bu)) \\
		& = (\by- \Pi_\disc y_\disc (\bu), \Pi_\disc(p_\disc(\bu)-p_\disc^*(\bu))  ) \\
		& = (\by- \by_\mesh, \Pi_\disc(p_\disc(\bu)-p_\disc^*(\bu))  ) \\
		& \qquad + (\by_\mesh- \Pi_\disc y_\disc (\bu), \Pi_\disc(p_\disc(\bu)-p_\disc^*(\bu))  )\\
		&\lesssim{}
		h^2\norm{\by}{H^2(\O)} \norm{\Pi_\disc(p_\disc(\bu)-p_\disc^*(\bu))}{}\\ 
		& \qquad + \norm{\by_\mesh-\Pi_\disc y_\disc(\bu)}{} \norm{\Pi_\disc(p_\disc(\bu)-p_\disc^*(\bu))}{}.  \end{align*} 
	A use of \eqref{T1.adj} and $\assum{1}$-i) leads to
		$\norm{\Pi_\disc p_\disc(\bu)-\Pi_\disc p_\disc^*(\bu)}{} \lesssim h^2\norm{\by}{H^2(\O)} + h^2\norm{f+\bu}{H^1(\O)}$.
	Plugged into \eqref{a1.inequality}, this estimate yields
	\begin{align}
	T_1 \lesssim  h^2(\norm{\by -\yd}{H^1(\O)}+  \norm{\by}{H^2(\O)}
	+\norm{f+\bu}{H^1(\O)} ). \label{a1new.ineq}
	\end{align}
	
	\textbf{Step 2}: estimate of $T_2$.
	
	Subtract the equations  \eqref{aux2} satisfied by $p_\disc^*(\bu)$
	and $p_\disc^*(\hat{u})$ to obtain, for all $z_\disc \in X_{\disc}$,
	\be \label{subtract_aux2}
	a_\disc(z_\disc, p_\disc^*(\bu) -p_\disc^*(\hat{u}))=(\Pi_\disc y_\disc(\bu) -\Pi_\disc y_\disc(\hat{u}),\Pi_\disc z_\disc).
	\ee
Since $ p_\disc^*(\hat{u}) - p_\disc^*(\bu) \in X_{\disc,\star}$, a use of Proposition \ref{prop.stab} in \eqref{subtract_aux2} yields
	\begin{align}
	T_2 &=\norm{\Pi_\disc p_\disc^*(\bu) -\Pi_\disc p_\disc^*(\hat{u})}{}
	\lesssim \norm{\Pi_\disc y_\disc(\bu) -\Pi_\disc y_\disc(\hat{u})}{}. \label{est.A2}
	\end{align}
	Choose $z_\disc=y_\disc(\bu)-y_\disc(\hat{u})$ in \eqref{subtract_aux2}, subtract the equations \eqref{state_aux} satisfied by $y_\disc(\bu)$ and $y_\disc(\hat{u})$, since $p_\disc^*(\bu)-p_\disc^*(\hat{u})\in X_{\disc,\star}$, to obtain
	\begin{align}
	\Vert\Pi_\disc( y_\disc(\bu)-y_\disc(\hat{u}))\Vert^2
	={}&a_\disc(y_\disc(\bu)-y_\disc(\hat{u}),p_\disc^*(\bu)-p_\disc^*(\hat{u}))\nonumber \\
	={}& (\bu-\hat{u}, \Pi_\disc p_\disc^*(\bu)-\Pi_\disc p_\disc^*(\hat{u})). \nonumber
	\end{align}
Set $w_\disc=p_\disc^*(\bu)-p_\disc^*(\hat{u})$, use orthogonality of $\Proj$, Cauchy-Schwarz inequality and \assum{2} to infer
		\begin{align}
		\Vert\Pi_\disc( y_\disc(\bu)-y_\disc(\hat{u}))\Vert^2 
		={}& (\bu-\hat{u},\Pi_\disc w_\disc) \nonumber \\
		={}& (\bu-\Proj\bu,\Pi_\disc w_\disc-\Proj(\Pi_\disc w_\disc))
		+ (\Proj\bu-\hat{u},\Pi_\disc w_\disc) \nonumber \\
		\lesssim_{\eta}{}& h \norm{\bu}{H^1(\O)}h\norm{\nabla_\disc w_\disc}{} + \int_{\O_{1,\mesh}}(\Proj\bu-\hat{u})\Pi_\disc w_\disc\d\x  \nonumber \\
	&+\int_{\O_{2,\mesh}}(\Proj\bu-\hat{u})\Pi_\disc w_\disc\d\x.\label{new.A2}
		\end{align}
		Equation \eqref{subtract_aux2} and the stability of the GDM (Proposition \ref{prop.stab}) show that
		\be \label{stability_aux2}
		\norm{\nabla_\disc w_\disc}{}= \norm{\nabla_\disc(p_\disc^*(\bu)-p_\disc^*(\hat{u}))}{} \lesssim \norm{\Pi_\disc( y_\disc(\bu)-y_\disc(\hat{u}))}{}.
		\ee
		A use of Holder's inequality, $\assum{4}$, $\assum{3}$, the fact that $w_\disc \in X_{\disc,\star}$ and \eqref{stability_aux2} yields an estimate for the second term on the right hand side of \eqref{new.A2} as follows:
		\begin{align}
	\int_{\O_{1,\mesh}}(\Proj\bu-\hat{u})\Pi_\disc w_\disc\d\x  \le{}& \norm{\Proj\bu-\hat{u}}{L^2(\O_{1,\mesh})}\norm{\Pi_\disc w_\disc}{L^2(\O_{1,\mesh})} \nonumber \\
		\le{}& h \norm{\bu}{W^{1,\infty}(\mesh_1)}|\O_{1,\mesh}|^\frac{1}{2} \norm{\Pi_\disc w_\disc}{L^{2^*}(\O)}|\O_{1,\mesh}|^{\frac{1}{2}-\frac{1}{2^*}} \nonumber \\
		\lesssim{}&   h^{2-\frac{1}{2^*}}\norm{\bu}{W^{1,\infty}(\mesh_1)}\norm{w_\disc}{\disc} \nonumber\\
		={}&   h^{2-\frac{1}{2^*}}\norm{\bu}{W^{1,\infty}(\mesh_1)}\norm{\nabla_\disc w_\disc}{} \nonumber\\
		\lesssim{}&   h^{2-\frac{1}{2^*}}\norm{\bu}{W^{1,\infty}(\mesh_1)}\norm{\Pi_\disc( y_\disc(\bu)-y_\disc(\hat{u}))}{}. \label{a21.ineq}
		\end{align}
		Consider now the last term on the right hand side of \eqref{new.A2}. For any $K\in \mesh_2$, we have $\bu=a$ on $K$, $\bu=b$ on $K$,
		or, by \eqref{proj_cont}, $\bu=-\alpha^{-1}\bp+\oc$ on $K$.
		Hence, on $K$, $\Proj\bu-\hat{u}=0$ or $\Proj\bu-\hat{u}=\alpha^{-1}\left(\hat{p}-\Proj\bp\right)$. This leads to $|\Proj\bu-\hat{u}|\le \alpha^{-1}|\hat{p}-\Proj\bp|$ on $\O_{2,\mesh}$.
		 Use \eqref{approx.xK}, the definition of $C_\disc$, the fact that $w_\disc \in X_{\disc,\star}$ and \eqref{stability_aux2} to obtain
		\begin{align}
		\int_{\O_{2,\mesh}}(\Proj\bu-\hat{u})\Pi_\disc w_\disc\d\x {}&\le \norm{\Proj\bu-\hat{u}}{L^2(\O_{2,\mesh})}\norm{\Pi_\disc w_\disc}{} \nonumber\\
		{}&\le\alpha^{-1}\norm{\Proj\bp-\hat{p}}{L^2(\O_{2,\mesh})}\norm{\Pi_\disc w_\disc}{} \nonumber\\
		{}&\lesssim_{\eta} h^2\alpha^{-1}\norm{\bp}{H^2(\O_{2,\mesh})}\norm{\nabla_\disc w_\disc}{}\nonumber \\
		{}&\lesssim_{\eta} h^2 \alpha^{-1}\norm{\bp}{H^2(\O_{2,\mesh})}\norm{\Pi_\disc( y_\disc(\bu)-y_\disc(\hat{u}))}{}. \label{a22.ineq}
		\end{align}
		Plug \eqref{stability_aux2}, \eqref{a21.ineq} and \eqref{a22.ineq} into \eqref{new.A2} and then in \eqref{est.A2} to get
	\be
	T_2\lesssim_{\eta}h^{2-\frac{1}{2^*}}\norm{\bu}{W^{1,\infty}(\mesh_1)}+ h^2 \left( \norm{\bu}{H^1(\O)} +\alpha^{-1}\norm{\bp}{H^2(\O_{2,\mesh})}\right). \label{yd.ineq} 
	\ee
	
	\medskip
	
	\textbf{Step 3}: estimate of $T_3$.\\
	Subtract \eqref{discrete_adjoint} from \eqref{aux2} with $g=\hat{u}$ and \eqref{discrete_state} from \eqref{state_aux} with $\hat{u}$ instead of $\bu$, we obtain for all $z_\disc \in X_{\disc}$ and $w_\disc \in X_{\disc,\star}$,
\be \label{t3.aux1}
	a_\disc(z_\disc, p_\disc^*(\hat{u})-\bp_\disc) = (\Pi_\disc y_\disc(\hat{u})-\Pi_\disc \by_\disc, \Pi_\disc z_\disc),
\ee
	\be\label{t3.aux2}
	a_\disc(y_\disc(\hat{u})-\by_\disc, w_\disc) = (\hat{u}-\bu_h, \Pi_\disc w_\disc).
	\ee
 Substitute $z_\disc = p_\disc^*(\hat{u})-\bp_\disc \in X_{\disc,\star}$ in \eqref{t3.aux1}, $w_\disc = y_\disc(\hat{u})-\by_\disc \in X_{\disc,\star}$ in \eqref{t3.aux2} and use Proposition \ref{prop.stab} to obtain
	\be\label{a3.ineq}
	T_3 =\norm{\Pi_\disc p_\disc^*(\hat{u})-\Pi_\disc \bp_\disc}{}  \lesssim \norm{\Pi_\disc y_\disc(\hat{u})-\Pi_\disc \by_\disc}{} \lesssim \norm{\hat{u}-\bu_h}{}.
	\ee
The optimality condition \eqref{opt_cont} \cite[Lemma 3.5]{CMAR} yields for a.e.\ $\x \in \O$,
		\[
		\big(\bp(\x)+\alpha\bu(\x)\big)\,\big(v(\x)-\bu(\x)\big) \geq 0 \mbox{   for all   } v \in  \Uad. 
		\]
		Since $\bu$, $\bp$ and $\bu_h$ are continuous at the centroid $\overline{\x}_K$, we
		can choose $\x=\overline{\x}_K$ and $v(\overline{\x}_K)=\bu_h(\overline{\x}_K)(=\bu_h$ on $K$). All the involved
		functions being constants over $K$, this gives
		\[
		\left(\hat{p}+\alpha\hat{u}\right) \left(\bu_h-\hat{u}\right) \geq 0 \mbox{ on $K$, for all $K\in\mesh$}.
		\]
		Integrate over $K$ and sum over $K \in \mesh$ to deduce
		\[
		(\hat{p} +\alpha \hat{u} ,\bu_h-\hat{u}) \geq 0.
		\]
		Choose $v_h=\hat{u}$ in the discrete optimality condition \eqref{opt_discrete} to establish
		\[(\Pi_\disc\bp_\disc+\alpha\bu_h,\hat{u}-\bu_h) \geq 0. \]
		Add the above two inequalities to obtain
		\[(\hat{p}-\Pi_\disc\bp_\disc+\alpha(\hat{u}-\bu_h),\bu_h-\hat{u}) \geq 0,
		\]
		and thus
		\begin{align}
		\alpha \norm{\hat{u}-\bu_h}{}^2 \leq{}& (\hat{p}-\Pi_\disc\bp_\disc,\bu_h-\hat{u}) \nonumber \\
		={}&(\hat{p}-\bp_\mesh,\bu_h-\hat{u})+(\bp_\mesh-\Pi_\disc p_\disc^*(\hat{u}),\bu_h-\hat{u})\nonumber \\
		& +(\Pi_\disc p_\disc^*(\hat{u})-\Pi_\disc\bp_\disc,\bu_h-\hat{u}) \nonumber \\
		=:{}&M_1+M_2+M_3.\label{a3new.ineq}
		\end{align}
		Since $\bu_h-\hat{u}$ is piecewise constant on $\mesh$,
		the orthogonality property of $\Proj$, \eqref{approx.xK} and \eqref{prop.M.2}
		in $\assum{1}$-ii) lead to
		\begin{align}
		M_1&= (\hat{p}-\Proj\bp_\mesh,\bu_h-\hat{u}) 
		\nonumber \\
&=  (\hat{p}-\Proj\bp,\bu_h-\hat{u})+ (\Proj(\bp-\bp_\mesh),\bu_h-\hat{u}) 
\lesssim_{\eta}  h^2\norm{\bp}{H^2(\O)}\norm{\bu_h-\hat{u}}{}. \label{m1.ineq}
		\end{align}	 
		By Cauchy--Schwarz inequality, triangle inequality and the notations in \eqref{ineq}, we obtain
		\be
		M_2\le \norm{\bp_\mesh-\Pi_\disc p_\disc^*(\hat{u})}{}\norm{\bu_h-\hat{u}}{}
		\lesssim (T_1+T_2) \norm{\bu_h-\hat{u}}{}.\label{m2.ineq}
		\ee
		Subtract the equations \eqref{discrete_state} and \eqref{state_aux} (with
		$\hat{u}$ instead of $\bu$) satisfied by $\by_\disc$ and $y_\disc(\hat{u})$, choose $w_\disc= p_\disc^*(\hat{u})-\bp_\disc$, and use the equations
		\eqref{discrete_adjoint} and \eqref{aux2} on $\bp_\disc$ and $p_\disc^*(\hat{u})$ to arrive at
\begin{align}
M_3	&={}(\Pi_\disc (p_\disc^*(\hat{u})-\bp_\disc),\bu_h-\hat{u})= a_\disc(\by_\disc-y_\disc(\hat{u}),p_\disc^*(\hat{u})-\bp_\disc) \nonumber\\
&={} (\Pi_\disc(y_\disc(\hat{u})-\by_\disc),\Pi_\disc(\by_\disc-y_\disc(\hat{u}))
\le 0. \label{m3.ineq}
\end{align}
		A substitution of \eqref{m1.ineq}--\eqref{m3.ineq} (together with the estimates
		\eqref{a1new.ineq} and \eqref{yd.ineq} of $T_1$ and $T_2$) into \eqref{a3new.ineq} yields an estimate on $\norm{\bu_h-\hat{u}}{}$ which, when
		plugged into \eqref{a3.ineq}, gives
	\be
	\begin{aligned}
		T_{3} \lesssim_{\eta}{}& \alpha^{-1}
		h^{2-\frac{1}{2^*}}\norm{\bu}{W^{1,\infty}(\mesh_1)}\\
		&+\alpha^{-1}h^2\big(\norm{\by -\yd}{H^1(\O)}+  \norm{\by}{H^2(\O)}+
		(1+\alpha^{-1})\norm{\bp}{H^2(\O)}\\
		&\quad\quad\quad\quad+\norm{f+\bu}{H^1(\O)}+\norm{\bu}{H^1(\O)}\big). \label{a3new1.ineq}
	\end{aligned}
	\ee
\textbf{Step 4}: conclusion.

A use of \eqref{def:cost} and the fact that $\bu$ is optimal leads to
$$\frac{\alpha}{2} \norm{\bu}{}^2 \le J(\by,\bu) \le J(y(0),0)=\frac{1}{2}\norm{y(0)-\yd}{}^2,$$
where $y(0)$ is the solution to the state equation \eqref{state1} with $u=0$. Hence,
\be \label{u.l2est}
\norm{\bu}{} \lesssim \sqrt{\alpha}^{-1}\left(\norm{f}{}+\norm{\by_d}{}\right).
\ee
From \eqref{proj_cont} and \eqref{def_proj[a,b]}, we have
$$\nabla\bu=\nabla P_{[a,b]}(-\alpha^{-1}\bp + \oc)=\mathbbm{1}_{(-\alpha^{-1}\bp + \oc)\in [a,b]}\nabla( -\alpha^{-1}\bp + \oc).$$
Note that $|\nabla( -\alpha^{-1}\bp + \oc)|=\alpha^{-1}|\nabla \bp|$. Therefore,
\be\label{u.nabla.l2est}
\begin{aligned}
\norm{\nabla\bu}{}^2=\int_{\O}^{}|\nabla\bu|^2\d\x={}&\int_{\O}^{}|\mathbbm{1}_{(-\alpha^{-1}\bp + \oc)\in [a,b]}\nabla( -\alpha^{-1}\bp + \oc)|^2\d\x \\
\lesssim{}& \alpha^{-2} \norm{\nabla \bp}{}^2.
\end{aligned}
\ee
Combine \eqref{u.l2est} and \eqref{u.nabla.l2est} to obtain
\be \label{u.h1est}
\norm{\bu}{H^1(\O)}\lesssim \sqrt{\alpha}^{-1}\left(\norm{f}{}+\norm{\by_d}{}\right) + \alpha^{-1} \norm{\nabla \bp}{}.
\ee
Use \eqref{u.h1est} in \eqref{a1new.ineq}, \eqref{yd.ineq} and \eqref{a3new1.ineq} and plug the resulting estimates in \eqref{ineq} to complete the proof.
\end{proof}

\begin{proof}[Proof of Theorem \ref{thm.fullsuper-convergence}]
	The proof of this theorem is identical to the proof of Theorem \ref{thm.super-convergence}, except for the estimate of $T_{2}$. This estimate is the only source of the $2-\frac{1}{2^*}$ power (instead of 2),
	and the only place where we used Assumption \assum{3}, here replaced by \eqref{Linfty.est}. 
Recall \assum{4} and use \eqref{Linfty.est} in \eqref{subtract_aux2}
		satisfied by  $p_\disc^*(\bu)-p_\disc^*(\hat{u})$ to write
		\begin{align}
		\int_{\O_{1,\mesh}}(\Proj\bu-\hat{u}){}&\Pi_\disc w_\disc\d\x  =\int_{\O_{1,\mesh}}(\Proj\bu-\hat{u})\big(\Pi_\disc p_\disc^*(\bu)-\Pi_\disc p_\disc^*(\hat{u})\big)\d\x  
		\nonumber\\
		&\lesssim \norm{\Proj\bu-\hat{u}}{L^\infty(\O_{1,\mesh})} \norm{\Pi_\disc p_\disc^*(\bu)-\Pi_\disc p_\disc^*(\hat{u})}{L^\infty(\O_{1,\mesh})}|\O_{1,\mesh}|  \nonumber\\
		&\lesssim h^2 \norm{\bu}{W^{1,\infty}(\mesh_1)}\norm{\Pi_\disc p_\disc^*(\bu)-\Pi_\disc p_\disc^*(\hat{u})}{L^\infty(\O)}  \nonumber\\
		&\lesssim h^2 \norm{\bu}{W^{1,\infty}(\mesh_1)}\delta \norm{\Pi_\disc y_\disc(\bu)-\Pi_\disc y_\disc(\hat{u})}{}.
		\label{improved.A2}
		\end{align}
		The rest of the proof follows from this estimate.
\end{proof}

\begin{proof}[Proof of Corollary \ref{cor.super-convergence}]
		The result for the state and adjoint variables can be derived exactly as in \cite[Corollary 3.7]{GDM_control}. 
\end{proof}

\thanks{\textbf{Acknowledgement}: The first author acknowledges the funding support from the Australian Government through the Australian Research Council's Discovery Projects funding scheme (pro\-ject number DP170100605). 
The second and third authors acknowledge the funding support from the DST project SR/S4/MS/808/12.}

\bibliographystyle{abbrv}
\bibliography{control_pure_neumann}								
\end{document}

%% file: fig-mlpunnc.pdf_t
\begin{picture}(0,0)%
\includegraphics{fig-mlpunnc.pdf}%
\end{picture}%
\setlength{\unitlength}{3947sp}%
\begingroup\makeatletter\ifx\SetFigFont\undefined%
\gdef\SetFigFont#1#2#3#4#5{%
  \reset@font\fontsize{#1}{#2pt}%
  \fontfamily{#3}\fontseries{#4}\fontshape{#5}%
  \selectfont}%
\fi\endgroup%
\begin{picture}(2841,1761)(4427,-2747)
\put(5825,-1574){\makebox(0,0)[lb]{\smash{{\SetFigFont{10}{12.0}{\familydefault}{\mddefault}{\updefault}{\color[rgb]{0,0,0}$D_\edge$}%
}}}}
\put(6011,-1775){\makebox(0,0)[lb]{\smash{{\SetFigFont{10}{12.0}{\familydefault}{\mddefault}{\updefault}{\color[rgb]{0,0,0}$\edge$}%
}}}}
\put(4940,-2577){\makebox(0,0)[lb]{\smash{{\SetFigFont{10}{12.0}{\familydefault}{\mddefault}{\updefault}{\color[rgb]{0,0,0}$\edge'$}%
}}}}
\put(4560,-1123){\makebox(0,0)[lb]{\smash{{\SetFigFont{10}{12.0}{\familydefault}{\mddefault}{\updefault}{\color[rgb]{0,0,0}$\dr\O$}%
}}}}
\put(5069,-2302){\makebox(0,0)[lb]{\smash{{\SetFigFont{10}{12.0}{\familydefault}{\mddefault}{\updefault}{\color[rgb]{0,0,0}$D_{\edge'}$}%
}}}}
\end{picture}%

%% file: fig-mesh.pdf_t
\begin{picture}(0,0)%
\includegraphics{fig-mesh.pdf}%
\end{picture}%
\setlength{\unitlength}{3947sp}%
\begingroup\makeatletter\ifx\SetFigFont\undefined%
\gdef\SetFigFont#1#2#3#4#5{%
  \reset@font\fontsize{#1}{#2pt}%
  \fontfamily{#3}\fontseries{#4}\fontshape{#5}%
  \selectfont}%
\fi\endgroup%
\begin{picture}(1974,1703)(5089,-2802)
\put(6826,-1261){\makebox(0,0)[lb]{\smash{{\SetFigFont{10}{12.0}{\familydefault}{\mddefault}{\updefault}{\color[rgb]{0,0,0}$K$}%
}}}}
\put(6392,-2733){\makebox(0,0)[lb]{\smash{{\SetFigFont{10}{12.0}{\familydefault}{\mddefault}{\updefault}{\color[rgb]{0,0,0}$\bfn_{K,\edge}$}%
}}}}
\put(6589,-2311){\makebox(0,0)[lb]{\smash{{\SetFigFont{10}{12.0}{\familydefault}{\mddefault}{\updefault}{\color[rgb]{0,0,0}$\overline{x}_\edge$}%
}}}}
\put(6882,-2092){\makebox(0,0)[lb]{\smash{{\SetFigFont{10}{12.0}{\familydefault}{\mddefault}{\updefault}{\color[rgb]{0,0,0}$\edge$}%
}}}}
\put(5955,-1720){\makebox(0,0)[lb]{\smash{{\SetFigFont{10}{12.0}{\familydefault}{\mddefault}{\updefault}{\color[rgb]{0,0,0}$\x_K$}%
}}}}
\put(6078,-2093){\makebox(0,0)[lb]{\smash{{\SetFigFont{10}{12.0}{\familydefault}{\mddefault}{\updefault}{\color[rgb]{0,0,0}$D_{K,\edge}$}%
}}}}
\end{picture}%